\documentclass[12pt]{amsart}
\usepackage{amssymb}
\usepackage{booktabs}
\usepackage{mathrsfs}
\usepackage{psfrag}
\usepackage{bm}
\usepackage{fourier}
\usepackage{enumitem}
\usepackage{fourier} 
\usepackage{comment}



\title[Brjuno-Like Functions]{Brjuno-Like Functions for nonlinear expanding maps: \\ Fractional Derivatives and Regularity Dichotomies}

\author[Marmi and Smania]{S. Marmi and  D. Smania}
\address{Scuola Normale Superiore \\  Piazza dei Cavalieri, 7 \\  56126 Pisa PI, Italy} 
\email{stefano.marmi@sns.it}
\urladdr{\href{https://homepage.sns.it/marmi/}{https://homepage.sns.it/marmi/}}
\address{Departamento de Matem\'atica \\
   ICMC/USP \\
               Avenida Trabalhador S\~ao-carlense, 400 \\ 
CEP: 13566-590 - S\~ao Carlos - SP \\ Brazil.}
              
\email{smania@icmc.usp.br} \urladdr{\href{http://www.icmc.usp.br/pessoas/smania/}{http://www.icmc.usp.br/pessoas/smania/}}
\thanks{D.S. was partially supported  by the São Paulo Research Foundation (FAPESP), Brasil,
Process Number 2017/06463-3, Bolsa de Produtividade em Pesquisa CNPq-Brazil 311916/2023-
6. S. M and D.S. were partially supported by the   project  “Dynamics and Information
Research Institute - Quantum Information (Teoria dell'Informazione), Quantum Technologies” within the agreement between UniCredit Bank and Scuola Normale Superiore.}

\newtheorem{theorem}{Theorem}[section]

\newtheorem{mainthm}{Theorem}

\newtheorem{lemma}[theorem]{Lemma}
\newtheorem{proposition}[theorem]{Proposition}

\theoremstyle{definition}

\newtheorem{remark}[theorem]{Remark}

\usepackage{constants} 
\newcommand{\secdot}[1]{\arabic{#1}}
\newconstantfamily{c}{
symbol=\lambda,
format=\secdot,
}
\renewconstantfamily{normal}{
symbol=C,
format=\secdot,
}

\newconstantfamily{e}{
symbol=\theta,
format=\arabic,
}

\newconstantfamily{A}{
symbol=A,
format=\arabic,
}

\newconstantfamily{G}{
symbol=G,
format=\arabic,
}

\newcommand{\Cll}[2][normal]{{\color{blue} \hypertarget{#2}{\Cl[#1]{#2}} }}
\newcommand{\Crr}[1]{\hyperlink{#1}{\Cr{#1}}}

\usepackage[colorlinks]{hyperref}

\newcounter{change}

\usepackage[disable, colorinlistoftodos, textwidth=4cm]{todonotes}
\newcommand{\change}[2][]{\refstepcounter{change} \todo[linecolor=blue,backgroundcolor=blue!25,bordercolor=blue,size=tiny,#1]{Note \thechange - #2}}

\makeatletter
\providecommand\@dotsep{5}
\renewcommand{\listoftodos}[1][\@todonotes@todolistname]{%
  \@starttoc{tdo}{#1}}
\makeatother

\hypersetup{ colorlinks=true, pdftitle={Brjuno-Like Functions for nonlinear expanding maps:  Fractional Derivatives and Regularity Dichotomies}, pdfsubject={ 37C30, 37D20,  37C40,  37E10,   37F50,  30H25, 37D35,   37A50, 37E05,     60F05}, pdfauthor={Stefano Marmi, Daniel Smania},pdfkeywords={ Brjuno function, Weierstrass function, Takagi function, Central Limit Theorem, expanding maps, thermodynamical formalism, cohomological equation, Livsic equation}}

\begin{document}

\begin{abstract} Cohomological equations appear frequently in dynamical systems.
One of the most classical examples is the Liv\v{s}ic equation
\[
v(x) = \alpha \circ F(x) - \alpha(x).
\]

The existence and regularity of its solutions $\alpha$ is well understood when $F$ is a hyperbolic dynamical system (for instance an expanding map of the circle) and $v$ is a H\"older function.

The \emph{twisted cohomological equation}
\[
v(x) = \alpha \circ F(x) - (DF(x))^\beta \, \alpha(x)
\]
is much less well understood. Functions similar to the famous Brjuno, Weierstrass, and Takagi functions
appear as solutions of this equation.
This functional equation also appears in the work of M.~Lyubich, and of
 Avila,  Lyubich, and  de~Melo in their study of deformations of
quadratic-like and real-analytic maps.

Nevertheless, there are some striking results concerning the (lack of)
regularity of solutions $\alpha$ when $F$ is a linear endomorphism of the circle and $v$
is very regular. Notable contributions include works by Berry and Lewis;
Ledrappier; Przytycki and Urba\'nski, and more recently by Bara\'nski,
B\'ar\'any and Romanowska, as well as by Shen, and by Ren and Shen, on
Takagi and Weierstrass (and Weierstrass-like) functions.

We study the regularity of solutions $\alpha$ when $F$ is a
\emph{nonlinear} expanding map of the circle and $v$ is not differentiable or even continuous,
a setting in which previously used transversality techniques do not appear to be applicable.
The new approach uses fractional derivatives to \emph{reduce}
the study of the twisted cohomological equation to that of a corresponding
Liv\v{s}ic cohomological equation, and to show that the resulting distributional
solutions (in the sense of Schwartz) satisfy certain Central Limit Theorem. 

  \end{abstract}

\maketitle

\makeatletter
\renewcommand{\l@section}{%
  \@tocline{1}{0pt}{0pt}{0em}{}%
}
\makeatother

\setcounter{tocdepth}{3}

\newpage
\tableofcontents

\section{Introduction}

\subsection{Twisted  and Lvisic Cohomological equations }

Cohomological equations are widerspread in dynamical systems.  The classical {\it Lvisic equation }
\begin{equation}\label{lvisic} v= \alpha\circ F - \alpha\end{equation} 
appears quite early in the study of Anosov diffeomorphisms and expanding maps \cite{L1972}.  Indeed we can see   the Liv\v{s}ic equation Liv\v{s}ic as  a particular case  the    {\it twisted cohomological equation} 
\begin{equation}\label{tce} v(x) =\alpha(F(x))  - DF(x)^\beta \alpha(x),\end{equation} 
where $v$ is a piecewise smooth function and $F$ is a piecewise smooth  map acting on either an interval or $\mathbb{S}^1$, and $\beta \in  \mathbb{R}.$ We describe below  many  instances  in the literature where such equation appears. One can ask about  the {\it existence, uniqueness and regularity of the solutions $\alpha$.}  We are going to see that this kind of problem appears in many branches of mathematics, in particular  when $F$ is a (piecewise) expanding Markovian map.  We call $\alpha$ a {\it Brjuno-like function}  due its similarity with  the Brjuno function. Shevchenko and Yampolsky \cite{sy} used the same  term for a similar class of functions. 

Weierstrass functions $W_{a,b}$, with $b \in \mathbb{N}\setminus \{0,1\}$ and $ab\geq 1$  and the Takagi function can be obtained in that way taking $$F(x)= bx \ mod \ 1.$$
Indeed if   $F$ is a linear endomorphism of the circle there is a  very long and exciting literature.

A common characteristic of these equations is the presence of a {\it dichotomy} in the regularity of the solutions $\alpha$: either the solution is highly irregular (or may not even exist), or it is very smooth. 

 It turns out that in many settings, even when $v$ is {\it very regular}, for a {\it typical} choice of $v$ the solution $\alpha$  has a very {\it wild} behaviour, usually manifested in poor differentiability properties, graphs with Hausdorff dimension larger than one, and related phenomena. A striking recent example is provided by Ren, and Shen~\cite{MR4337976} on Weierstrass-type functions.

There are far fewer results for {\it nonlinear} maps $F$, which will be the main focus of this work.

 The case $\beta=1$ is of particular interest in the study of deformations of dynamical systems. The study of the solutions of this equation plays a proeminent part on the works of Lyubich \cite{lyubich} and Avila, Lyubich and de Melo \cite{alm} on the dynamics of quadratic-like maps and real-analytic maps, as well as in the study of deformations of piecewise expanding maps on the interval by Ragazzo and S. \cite{rs22}. 
 
 Todorov \cite{todorov} studied the regularity of the solutions of (\ref{tce}) using both methods of de Lima and S. \cite{ls} and Ledrappier \cite{ledrappier2}.  See also Otani \cite{otani}

If $F$ is an $C^{2+\epsilon}$ expanding map on the circle and $v \in C^{1+\epsilon}$,    de Lima and S. \cite{ls} made an explicit connection between the regularity of the solution $\alpha$ for the twisted cohomological equation (\ref{tce}) when $\beta=1$ and   the  Lvisic equation (\ref{lvisic}). The equation 
\begin{equation}\label{def}  \frac{Dv+D^2F\circ \alpha}{DF} = \phi\circ F -\phi \end{equation} 
has a measurable solution $\phi$  if and only if  $\alpha$ \change{$v\mapsto \alpha$}is $C^{1+\epsilon}$. Furthermore such solution $\phi$  does not exist {\it if and only if}  $\alpha$ is nowhere differentiable. This gives an explicit connection between the twisted cohomological equation and the Lvisic cohomological  equation, as well as a nice  and simple characterization of the dichotomy  for the case $\beta=1$ even when  $F$ is nonlinear. 

Our goal is to provide a framework for the dichotomy phenomenon for solutions of the twisted cohomological equation that not only encompasses several known examples, but also includes new manifestations of the same phenomenon, such as cases in which $F$ is expanding but non-linear    $v$ has singularities.

In particular, our methods rely on a new  deep connection between the twisted cohomological equation and the Livšic equation through {\it fractional derivatives}, as well as on the ergodic properties of $F$. 
As in~\cite{ls}, the asymptotic variance and the Central Limit Theorem for certain observables plays a central role in the characterization of the dichotomy.

\subsection{Weierstrass, Takagi function and friends}

Perhaps surprisingly, the first examples of H\"older-continuous functions with wild behavior are {\it explicitly constructed}. In 1872, Weierstrass introduced a remarkable family of functions that are continuous everywhere but differentiable nowhere:
\begin{equation}\label{wei}
W_{a,b}(x) = \sum_{k=0}^{\infty} a^k \cos(b^k \pi x).
\end{equation}
Here, $a$ is a real number with $0 < a < 1$, and $b$ is an odd integer. Weierstrass assumed that the  parameters must satisfy the condition $ab > 1 + \frac{3\pi}{2}$. Later, in 1916, G. H. Hardy \cite{hardy} proved that the function $W$ is continuous and nowhere differentiable whenever $0 < a < 1$ and $ab \ge 1$, with no requirement for $b$ to be an integer. Furthermore, when $ab > 1$, the function $W_{a,b}$ is H\"older continuous with exponent $-\log a / \log b$.

Since the introduction of the Weierstrass function, many other contributions have enriched this area. In 1903, T. Takagi \cite{takagi} (see also Allaart and Kawamura \cite{alaart}) presented another example of a continuous function that is nowhere differentiable, now known as the {\it Takagi function},  defined by
$$
T(x) = \sum_{k=0}^{\infty} \frac{1}{2^k} \inf_{m \in \mathbb{Z}} |2^k x - m|.
$$

The question of whether the Hausdorff dimension of the graph of the classical Weierstrass function is equal to $2 + \frac{\log a}{\log b}$ , posed by B. Mandelbrot,  has captivated many mathematicians. For  an integer  $b\geq 2$  and $1/b<a<1$ this is question (and a similar one for  Takagi-like functions) are connected with the dynamics of  $f_b(x) = bx mod 1$ and skew-products  associated  to it by  works of  Przytycki and Urba{\'n}ski \cite{fu} and Ledrappier \cite{ledrappier2},  that employs  deep  developments in the geometry of SRB measures by Ledrappier and Young \cite{ly}.

  Significant progress was made over the past recent decades. Notably, Barański, Bárány, and Balázs \cite{MR3255455}  introduced new methods  that   settle the question for a large subset of parameters.   A  final  positive answer was given for all parameters  by Shen \cite{MR3803788}. In particular the work of  Tsujii \cite{tsujii} also played a important role.  See also Ren and Shen \cite{MR4337976} for further developments.

Another way to see the irregularity of these functions is that their modulus of continuity is somehow {\it random}. See the results by  Gamkrelidze \cite{gamk1}\cite{gamk2} for $W_{b,b}$ and $T$, that was later extended for a large class of functions and non-linear expanding maps  by  de Lima and S.  \cite{ls}. 

\subsection{The Brjuno function}
The \emph{Brjuno function} is  a cocycle
 \cite{MMY2} under the action of the modular group $\mathrm{PGL}_2(\mathbb{Z})$ action on $\mathbb{R}\setminus\mathbb{Q}$ , see \cite[Appendix 5]{MMY3} as well as the solution of a twisted cohomological equation associated to the  \emph{Gauss map} $G:(0,1]\to [0,1]$,  $G(x) = \{1/x\}$, where $\{t\}$ is the fractional part of $t$. 
 
Let $x_n = G^n(\{x\})$, $\beta_{n}(x) = x_0x_{1}\cdots x_n$, $n\ge 0$, and $\beta_{-1}\equiv1$.
J.-C. Yoccoz \cite{yoccoz1995petits} defined
the \emph{Brjuno function} $B:\mathbb{R}\setminus\mathbb{Q} \to \mathbb{R}\cup\{\infty\}$  by
\begin{equation*}\label{eq:def Brjuno}
B(x) = - \sum_{n=0}^\infty \beta_{n-1}(x)\log{x_n},
\end{equation*}
The set of Brjuno numbers can be characterized as the set where $B(x)<\infty$. 
Indeed, there is a positive constant $C$ such that
\begin{equation}\label{eq:B-logq/q}
\left|B(x)-\sum\limits_{n=0}^{\infty}\dfrac{\log (Q_{n+1}(x))}{Q_n(x)}\right|<C.
\end{equation}

Brjuno numbers and the Brjuno function play a very important role in one-dimensional small divisor problems 
\cite{yoccoz1995petits, EKMY2002}: the former exactly characterize the local conjugacy classes in linearization problems, the latter even provides a deep insight on the (discontinuous) dependence on the rotation number of the size of the rotational domains, interpolating the radius of convergence of the linearization up to a continuous \cite{BC2006}, possibly $\frac{1}{2}$-H\"older continuous \cite{CC2015, MMY1} correction.

Yoccoz proved that
$$
B(x) = -\log x + x B\circ G(x) \; , \;\;\; x\in (0,1)\setminus\mathbb{Q}\; , \;\;\; B(x)=B(x+1) 
$$
thus
the set of Brjuno numbers is invariant under the modular group $\mathrm{PSL}_{2}(\mathbb{Z})$.

The above functional equation can also be interpreted as a {\it twisted} cohomological equation 
\begin{equation} \label{brjuno}  \frac{\log x}{x} = B\circ G(x) - |DG|^{1/2} B(x) ,\end{equation} 
associated to the Gauss map, where the twist factor is not constant but given by the multiplication by $x$. In \cite{MMY1} several results have been proven concerning the regularity of the (coboundary) operator $(Tf)(x)= x(f\circ G)(x)$ acting on various spaces of periodic functions: $L^p$, BMO (Bounded Mean Oscillation) and H\"older. It turns out that $T$ is a contraction in the first two cases, whereas in the H\"older case its behaviour changes according to whether  the exponent is greater or smaller than $\frac{1}{2}$. Similar results have been obtained for other representatives of Brjuno function obtained by replacing the Gauss map with other continued fraction maps with full branches, see \cite{MMY1, LM, BLM}. The Brjuno function and other generalizations play an important role also in some problems in analytic number theory (see, e.g., \cite{RR2013, M2017, P2017}), including the Nyman-Beurling criterion for the Riemann Hypothesis (see \cite{D2020} and references therein). The Brjuno function itself is ''close to being a quantum modular form of weight $-1$'' (\cite{BD2022}, p.3).

\subsection{Random Hölder Functions} Weierstrass and Takagi functions are defined in a rather algorithmic fashion; however, it is remarkable that they capture (in an appropriate sense) the \emph{typical} behaviour of H\"older functions.
Besicovitch and Ursell \cite{ursell} proved that the Hausdorff dimension of the graph of a $\beta$-H\"older continuous function is at most $2-\beta$ and that  this  upper bound  is attained for certain   functions similar to Weierstrass functions.  Instigated by early investigations by Mandelbrot \cite{mandelbrot} and  Berry and  Lewis \cite{bl} there are results on the sharpness of this upper bound  for randon Weierstrass functions  by  Hunt \cite{hunt}) and   for random wavelet representations of functions  by Roueff \cite{roueff}. 

Indeed {\it most}  Hölder functions on an interval \( I \)  exhibit highly irregular behavior  in the {\it measure-theoretical sense}. In fact,  Clausel and Nicolay \cite{cn}   (See also \cite{cn2}) proved that for a {\it prevalent set}  of  functions \( \alpha \) in the space of \( \beta \)-Hölder continuous functions \( C^\beta(I) \), with \( \beta \in (0,1) \), the following properties hold 

\begin{itemize}
    \item The function \( \alpha \) is nowhere differentiable. More precisely, for every \( x \in I \)
    \[
    \lim_{\delta \rightarrow 0^+} \sup_{0 < |y - x| < \delta} \frac{|\alpha(y) - \alpha(x)|}{\delta^\beta} > 0.
    \]
    \item The Hausdorff dimension of the graph of \( \alpha \) is equal to \( 2 - \beta \).
\end{itemize}

\section{Main results}

\subsection{Dichotomy for the regularity of  Brujno-like functions} 

Given an $C^1$ expanding map  of the circle $F\colon \mathbb{S}^1 \rightarrow \mathbb{S}^1$,  a continuous function $v\colon \mathbb{S}^1\rightarrow \mathbb{R}$ and $\beta\in (0,1]$ let $ \alpha= \alpha_{\beta, v}$ be the unique bounded function that is the solution of  (\ref{tce}).

\begin{mainthm}[Dichotomy  of Smoothness for H\"older data] \label{mainex}  Let $F\colon \mathbb{S}^1 \rightarrow \mathbb{S}^1$  be a $C^{1+\gamma}$   expanding map on the circle with degree $2$ and $\gamma\in (0,1)$.  Let $C^k(\mathbb{S}^1)$, with $k\geq \gamma$ be  the Banach space of real-valued  $C^k$ functions.  For $\beta\in (0,\gamma)$ let 
$$\Omega_\beta = \{ v\in C^k(\mathbb{S}^1) \colon  \alpha_{\beta, v}  \not\in  \bigcup_{\delta> 0} C^{\beta+\delta}(\mathbb{S}^1)\}.$$
  Then $\alpha_{\beta,v} \in C^\beta(\mathbb{S}^1)$ for every $v\in C^k(\mathbb{S}^1)$. Moreover 
\begin{itemize} 
\item[A.] The set $\Omega_\beta$ is an  open and dense  subset of $C^k(\mathbb{S}^1)$.
\item[B.]  For a residual set of functions
$$\eta\in C^s([0,1]^d, C^k(\mathbb{S}^1)$$
where $s\in \mathbb{N}$ and $d\geq 1$  we have that 
$$\{ t\in [0,1]^d\colon\eta(t)\in \Omega_\beta \}.$$
has full Lebesgue measure on $[0,1]^d$.
\item[C.]  For a real analytic family $t\in (a,b)\mapsto v_t\in C^k(\mathbb{S}^1)$ we have that 
$$\{t\in (a,b)\colon v_t\not\in \Omega_\beta\}$$
is   either $(a,b)$, or  it is made of isolated points (and in particular is countable).
\item[D.] Let $v\in C^k(\mathbb{S}^1)$. Let 
$$\Lambda = \{ \beta\in (0,\gamma)\colon  v \not\in \Omega_\beta\}.$$
Then either $\Lambda$ contains only  isolated points (in particular it  is countable), and $0$ is not an accumulation point of $\Lambda$,  or $\Lambda=(0,\gamma) $.  
\item[E.] If $v\not\in \Omega_\beta$ then $\alpha_{\beta,v}\in C^{\gamma}(\mathbb{S}^1)$. 
\item[F.] If $v\in \Omega_\beta$ then $\alpha_{\beta,v}$ is nowhere differentiable, satisfies  the $\beta$-anti-H\"older property  and 
$$HD(\{ (x,\alpha_{\beta,v}(x))\colon x\in \mathbb{S}^1   \})> 1.$$
where $HD$ denotes the Hausdorff dimension.
\end{itemize} 
\end{mainthm}

\begin{remark} The dichotomy $D$, together with  $E$-$F$,   was  motivated by the  the striking dichotomy result by Ren and Shen \cite{MR4337976} when $F = \ell x \ mod  1$  and $v$ is real-analytic on the circle. One of the main novelties of Theorem \ref{mainex} is that $F$ can be genuinely {\it nonlinear}; more precisely, $F$ need not be smoothly conjugate to $\ell x \bmod 1$. This is a new result even when $F$ and $v$ are real-analytic.  
\end{remark} 

\begin{remark}
We note that previous arguments based on transversality conditions, as
in Ledrappier \cite{ledrappier2},  Barański, Bárány, and Balázs \cite{MR3255455}, Ren and Shen \cite{MR4337976} and
  other authors, cannot be applied here, since $v$ may fail to be
differentiable (it can in fact be nowhere differentiable), and those
transversality conditions involve the derivative of $v$.
\end{remark}

\begin{remark}
The assumption that $F$ has degree $2$ is not really necessary; however, it makes the notation significantly simpler.
\end{remark}

The following result give us a more precise  statement when $f$ are $v$ are  more regular. 

\begin{mainthm}[Dichotomy  of Smoothness for $C^{1+\gamma}$ data] \label{mainex5}  Let $F\colon \mathbb{S}^1 \rightarrow \mathbb{S}^1$  be a $C^{2+\gamma}$   expanding map on the circle with degree $2$ and $\gamma\in (0,1)$.  Let $v\in C^{1+\gamma}(\mathbb{S}^1)$ be a real-valued function and 
$$\Lambda = \{ \beta\in (0,1)\colon  v \not\in \Omega_\beta\}.$$

Then either $\Lambda=(0,1)$ or $\Lambda$  is finite. Moreover for every $\beta \in   \Lambda$ we have that $\alpha_{v,\beta}$ is $C^{1+\gamma}.$

Additionally if   $$\frac{Dv+D^2F\cdot \alpha_{v,1}}{DF}$$
is  not cohomologous to a constant then    $\Lambda$ is finite.  
\end{mainthm}

Consider the Besov space of functions $B^s_{1,1}$ on $\mathbb{S}^1$, with $s\in (0,1)$.   Functions  $v\in B^s_{1,1}$ are are not even continuous and/or bounded  in general, and indeed  they can  have a  very bad behavior.  We obtain a new dichotomy for such functions, which are in some sense more similar to the Brjuno function.

\begin{mainthm}[Dichotomy of Smoothness for Besov data] \label{MTB}  Let $F\colon \mathbb{S}^1 \rightarrow \mathbb{S}^1$  be a $C^{1+\gamma}$   expanding map on the circle with degree $2$ with  $\gamma\in (0,1)$. Let $\beta\in (0,\gamma-1/2)$.    For a real-valued function $v\in B^\gamma_{1,1}$   let $\alpha_v$ be the only $L^1$ solution of  $(\ref{ch2})$. Let 

$$\Omega_\beta = \{v\in B^\gamma_{1,1} \colon  \alpha_v\not \in \bigcup_{\delta> 0} B^{\beta+\delta}_{1,1}  \}.$$

Then 
\begin{itemize} 
\item[A.] $\Omega$ is open and dense subset of $B^\gamma_{1,1}$. 
\item[B.] For a residual set of functions
$$\gamma\in C^k([0,1]^d, B^\gamma_{1,1})$$
where $k\in \mathbb{N}$, we have that 
$$\{ t\in [0,1]^d\colon\gamma(t)\in \Omega_\beta \}.$$
has full Lebesgue measure on $[0,1]^d$.
\item[C.] Let $t\in J\mapsto v_t \in B^\gamma_{1,1}$, where $J$ is an open  interval,  be a  real-analytic function. Let 
$$\Lambda = \{ t\in J\colon v_t\not\in \Omega_\beta  \}.$$
Then either $\Lambda$ contains only  isolated points (in particular it  is countable),  or $\Lambda=J $.
\item[D.] Let $v\in B^\gamma_{1,1}$  
and 
$$\Lambda = \{ \beta\in (0,\gamma-1/2) \colon   v \not\in \Omega_\beta\}.$$
Then either $\Lambda$ contains only  isolated points (in particular it  is countable),  or $\Lambda=(0,\gamma-1/2)$.
\item[E.] If $v\not\in \Omega_\beta$ then $\alpha_{\beta,v}\in B^{\gamma-\delta}_{1,1}$ for every $\delta > 0$. 
\end{itemize} 
\end{mainthm} 

\begin{remark}
Note that Theorem~\ref{MTB} does not allow us to take $F$ to be the Gauss map. In particular, it cannot be applied to the Brjuno function itself. One may wonder whether our  methods can be extended to cover this case.
\end{remark}

\subsection{Strategy of the proof} The  idea is to \emph{reduce} the study of the twisted equation 
\begin{equation} \label{tce4}
  v = \alpha\circ F - (DF)^\beta \alpha
\end{equation}
to the Liv\v{s}ic equation 
\begin{equation} \label{liv4}
  \phi = \psi\circ F - \psi
\end{equation}
by defining a suitably adapted {\it fractional derivative}  $D^\beta$. This
definition actually depends on the expanding map $F$. In the particular
case of the angle-doubling map $F(x)= 2x \bmod 1$, this fractional
derivative behaves quite nicely, mimicking the classical Chain and
Leibniz rules for the usual derivative:
\[
D^\beta (\theta\circ F) =( D^\beta \theta)\circ F \,(DF)^\beta,
\qquad
D^\beta (\theta |Df|^\beta) = |Df|^\beta \, D^\beta \theta.
\]
Thus, applying $D^\beta$ to \eqref{tce4} we obtain
\begin{equation}
  \frac{D^\beta v}{(DF)^\beta}  = (D^\beta \alpha)\circ F - D^\beta \alpha.
\end{equation}

Hence, instead of studying the regularity of $\alpha$ directly, we can
study the regularity of solutions of the Liv\v{s}ic equation
\eqref{liv4} with
\[
\phi =  \frac{D^\beta v}{(DF)^\beta}.
\]
Since $\phi$ is regular enough, it is well known that the existence of a
function $\psi$ solving \eqref{liv4} is governed by the asymptotic
variance $\sigma^2(\phi)$: such a solution exists if and only if
$\sigma^2(\phi)=0$, and in that case the solution $\psi$ inherits the
regularity of $\phi$. This implies that $\alpha$ is \emph{very}
regular. On the other hand, if $\sigma^2(\phi)> 0$ such a function
$\psi$ does not exist, but we can see that a \emph{formal solution} of
\eqref{liv4} is given by 
\[
\psi = -\sum_{k=0}^\infty \phi\circ F^k,
\]
that implies that   $D^\beta \alpha$ is a Birkhoff sum (modulo an additive constant). We can make this rigorous
by interpreting $\psi$ as a \emph{distributional} solution of
\eqref{liv4}. Again, since $\phi$ is regular enough, one can prove that
the (finite-time) Birkhoff sums behave quite wildly and, for instance,
satisfy a \emph{Central Limit Theorem}. This wild behaviour of
$D^\beta \alpha$ implies that $\alpha$ is as irregular as possible.

This yields our dichotomy, and characterizes it in terms of a
\emph{single} non-negative number $\sigma^2(\phi)$. Note that
$\phi = \phi_{v,\beta}$ depends on $v$ and $\beta$. Using transfer
operators and thermodynamical formalism we can show that
\[
v\mapsto \sigma^2(\phi_{v,\beta})
\quad\text{and}\quad
\beta \mapsto \sigma^2(\phi_{v,\beta})
\]
are \emph{real-analytic} maps, and then Theorems \ref{mainex} and
\ref{MTB}, as well as other dichotomy results, follow.

However, one major difficulty is that when $F$ is not linear the Chain
and Leibniz rules above do not hold \emph{exactly}. Nevertheless, we
will see that even when $\theta$ is a suitable \emph{distribution} we
have
\[
D^\beta (\theta\circ F) =( D^\beta \theta)\circ F \,(DF)^\beta + R_C(\theta),
\qquad
D^\beta (\theta |Df|^\beta) = |Df|^\beta \, D^\beta \theta + R_L(\theta),
\]
where $R_C$ and $R_L$ are linear transformations whose images lie in a
space of \emph{functions} rather than distributions. This allows us to
differentiate \eqref{tce4} as before and replace the previous
definition of $\phi$ by
\[
\phi = \frac{D^\beta v + R_C(\alpha) - R_L(\alpha)}{(DF)^\beta},
\]
and then carry out the argument as before.

\subsection{Markovian expanding map } 

Let 
$$F\colon I_0\cup I_1 \rightarrow I$$
be an {\bf onto Markovian expanding map}  acting  on  an compact metric space  $(I,d_I)$ with a reference measure $m$, with 
\begin{itemize}
\item $m(I\setminus (I_0\cup I_1)=0$, where $I_i$ are open sets. 
\item The map $F\colon  I_i \rightarrow  I$ is expanding, that is there is $\lambda > 1$ such that 
$$d_I(F(x),F(y))\geq \lambda d_I(x,y)$$
for every $x,y\in I_i$. 
\item The map $F\colon I_i \rightarrow I$ has a H\"older jacobian with respect to a reference measure $m$, that is,   there is a H\"older function $g\colon I \rightarrow \mathbb{R} $ such that $g> 1$ everywhere and 
$$\int_{I_i} \psi\circ F\ g\ dm = \int_I  \psi \ dm$$
for every $m$-integrable function $\psi$. 
\item $m(I\setminus F(I_i))=0$ for $i=1,2$.
\end{itemize} 

 Let $\mathcal{P}^0=\{I\}$, $\mathcal{P}^1=\{I_1,I_2\}$ and $\mathcal{P}^k$ be the induced Markov partition for $F^k$. 
If $Q\in \mathcal{P}^{k+1}$, $P\in \mathcal{P}^{k}$ and $Q\subset P$ we say that $P$ is the parent of $Q$ and $Q$ is a child of $P$.  

Define 
$$g_n(x) = \prod_{k=0}^{n-1} g(F^kx).$$

We will denote by $|A|$ the measure $m(A)$. It is also convenient to replace the metric $d_I$ by a metric $d$ defined by
\begin{equation}\label{metricdef}  d(x,y)= \inf_{x,y\in P\in \cup_k \mathcal{P}^k}  |P|.\end{equation} 
Note that $d$ not necessarily induces the same topology on $I$ but  $F$ is also expanding with respect to $d$ and $g$ is H\"older with respect to $d$. 

It is well-know that every such $F$ as a unique invariant probability $\mu << m$. Indeed $d\mu/dm$ is H\"older and positive, so in particular there is $C > 0 $ such that 
\begin{equation} \label{lower} \frac{1}{C}<  \frac{d\mu}{dm}< C.\end{equation}

\subsection{Besov spaces} To get sharp results for the regularity of the solutions of the twisted cohomolgical   equation it will be quite important to consider appropriated spaces  of functions  that will be defined using the Markov partition of $F$. 

Let $Q_1^P$ and $Q_2^P$ be the children of $P$, indexed in such way that if $P\in \mathcal{P}^k$ and $R\in \mathcal{P}^{k+1}$ and $F(R)=P$ then 
$F(Q_1^R)= Q_1^P$ and $F(Q_2^R)= Q_2^P$.
 
Define
$$\phi_P = \frac{1_{Q_1^P}}{|Q_1^P|} -\frac{1_{Q_2^P}}{|Q_2^P|}.$$
Note that

$$\mathcal{H}= \{1_I\}\cup \{\phi_P\colon \ P\in \cup_i \mathcal{P}^i  \}.$$
is, after proper normalization, a Hilbert basis of $L^2(I)$.  It is called a {\bf unbalanced Haar wavelet basis}. Indeed  Girardi  and Sweldens \cite{gw} and Haar \cite{haar} proved that $\mathcal{H}$ is a unconditional basis of $L^\rho(I)$ for every $\rho > 1$.

Given $s \in \mathbb{R}$, the space of formal series 
$$
\psi = c_0 1_I + \sum_k \sum_{P \in \mathcal{P}^k} c_P \, |P|^{s+1} \phi_P
$$
such that 
$$
|\psi|_{\mathcal{B}^s_{\infty,\infty}}
  = |c_0| + \sup_k \sup_{P \in \mathcal{P}^k} |c_P| < \infty
$$
is called the Besov space $\mathcal{B}^s_{\infty,\infty}$.  

It turns out that, for $s>0$, we have $\mathcal{B}^s_{\infty,\infty} = C^s(I,d)$, the space of all $s$-Hölder continuous functions with respect to the metric $d$ defined in (\ref{metricdef}). Indeed 

\begin{theorem}[Marra, Morelli, and S. \cite{mms}] \label{fbeta} Let $s > 0$.   There is $C >0$ such that the following holds. For every $\psi\in \mathcal{B}^s_{\infty,\infty}$ and $x,y \in Q\in \mathcal{P}^{k_0}$ we have
$$|\psi(x)-\psi(y)|\leq C|\psi|_{\mathcal{B}^s_{\infty,\infty}}|Q|^s.$$
Moreover  
$$ \psi 1_Q = m(\psi,Q) 1_Q + \sum_{k\geq k_0}  \sum_{\substack{P\in \mathcal{P}^k\\ P\subset Q} } c_P |P|^s \phi_P,$$
where 
$$m(\psi,Q)=\frac{1}{|Q|} \int_Q \psi \ dm,$$
and $|m(\psi,Q)|,|c_P|\leq |\psi|_{\mathcal{B}^s_{\infty,\infty}}.$
\end{theorem}

Moreover, the $s$-Hölder norm  and the norm $|\cdot|_{\mathcal{B}^s_{\infty,\infty}}$ are equivalent.   In particular if $h \in \mathcal{B}^s_{\infty,\infty}$ then $e^{\beta h}  \in \mathcal{B}^s_{\infty,\infty}$ for every $\beta \in \mathbb{R}$. 
For $s < 0$, the space $\mathcal{B}^s_{\infty,\infty}$ can be interpreted as a space of distributions. 

The formal series 
$$
\psi = d_0 1_I + \sum_k \sum_{P \in \mathcal{P}^k} d_P \, |P|^{s} \phi_P
$$
belongs to the Besov space $\mathcal{B}^s_{1,1}$ if 
$$
|\psi|_{\mathcal{B}^s_{1,1}}
  = |d_0| + \sum_k \sum_{P \in \mathcal{P}^k} |d_P| < \infty.
$$

For $s>0$, this is a function space continuously embedded in $L^1$.  
For $s \leq 0$, it can be viewed as a space of distributions, as introduced in  Marra, Morelli, and S. \cite{mms}.

 \subsection{Low regularity of the solutions of the twisted cohomologial equation}  We are ready to give {\it lower bounds } for the regularity of the solutions of the twited cohomological equation.   The following theorems generalize several results by M., Moussa and Yoccoz \cite{MMY1}, Baladi and S. \cite{alt}, Todorov \cite{todorov}, and  Bettin and Drappeau  \cite{BD2022} to our setting.
\begin{mainthm}[Solutions with low regularity for H\"older input] \label{main4} Given a bounded function $v$ there is a unique bounded solution $\alpha$ for the cohomological equation 
\begin{equation}\label{ch2} v= \alpha\circ F-  g^{\beta} \alpha, \end{equation} 
Moreover suppose $v\in \mathcal{B}^s_{\infty,\infty}$, and  $\log g \in \mathcal{B}^\gamma_{\infty,\infty}$. 
\begin{itemize} 
\item[I.] If \   $0< s < \operatorname{Re} \beta< \gamma.$   Then  $\alpha \in \mathcal{B}^s_{\infty,\infty}$ and 
 $$v\in \mathcal{B}^s_{\infty,\infty} \overset{S_{\gamma,\beta}}{\mapsto}  \alpha \in \mathcal{B}^s_{\infty,\infty}$$ 
is a bounded linear transformation. Moreover for $0 < s< \beta_0 < \gamma$ and $\delta > 0$ we have
$$\sup_{\substack{  \beta_0 \leq Re \ \beta< \gamma\\ | Im\ \beta|< \delta } }  \|  S_{\gamma,\beta}\|_{\mathcal{B}^s_{\infty,\infty}} < \infty. $$  
\item[II.] Let $0< \operatorname{Re} \beta=s  < \gamma$. Then there is $C$ such that 
$$|\alpha(y)-\alpha(x)|\leq  C(d(y,x)^{\operatorname{Re} \beta}(1-\ln d(x,y))$$
\item[III.] Let $0< \operatorname{Re} \beta<  s$ and $\operatorname{Re} \beta < \gamma$. Then $\alpha \in \mathcal{B}^{\operatorname{Re} \beta}_{\infty,\infty}$ and
$$v\in \mathcal{B}^s_{\infty,\infty}   \mapsto \alpha \in \mathcal{B}^{    \operatorname{Re} \beta}_{\infty,\infty}$$  
is a bounded linear transformation.  

\end{itemize} 
\end{mainthm}

\subsection{Distributional solutions of the  Liv\v{s}ic cohomological equation, asymptotic variance and the Central Limit Theorem} Taking $\beta = 0$ in the twisted cohomological equation, we recover the classical Livšic cohomological equation, a problem extensively studied in the setting of Markov expanding maps. We shall see that a careful analysis of the regularity of {\it distributional solutions}  of the Livšic equation plays a major role in the case $\beta \in (0,1)$.

Here  we  study of (distributional) solutions of the classical {\bf Liv\v{s}ic cohomological equation}  

$$\phi = \psi\circ F -\psi$$

Concerning solutions $\psi$  that are {\it functions,}  the case of H\"older observables $\phi$ is well understood: an $L^2$ solution $\psi$ exists precisely when $\sigma^2(\phi) = 0$, where $\sigma^2(\phi)$ is the {\bf  asymptotic variance }  of $\phi$.  Moreover such a solution is itself Hölder. 

 The asymptotic variance will play a essential role in this work.  Suppose that either $\phi \in \mathcal{B}^t_{1,1}$, with $t\geq 1/2$, or $\phi \in \mathcal{B}^s_{\infty,\infty}$, with$s> 0$.   Then  $\phi \in L^2(m)$. Suppose that 
$$\int \phi \ \rho dm=0.$$ 
Then we have spectral gap on the corresponding function space   and we can define the {\bf  asymptotic variance } 
\begin{equation}\label{variance} \sigma^2(\phi)=\lim_n \int \Big( \frac{\sum_{k<  n} \phi\circ F^k}{\sqrt{n}} \Big)^2 \rho \ dm.\end{equation}

However when we do not have a function $\psi$ satisfying the Liv\v{s}ic cohomological equation, we can yet consider {\it distributional solutions} and study its regularity, even with distributional initial data $v$.  We note that, previously, M., Moussa, and Yoccoz \cite{MMY3}
considered generalized solutions (in the sense of hyperfunctions) to the twisted cohomological equation associated with the Gauss map.

\begin{mainthm}[Distributional solutions for distributional initial data] \label{dic333} Let  $\log g \in \mathcal{B}^\gamma_{\infty,\infty}$. For $\hat{\beta}\in (-\gamma,\gamma)$ and  $\beta \in \mathbb{C}$  close enough  to $\hat{\beta}$, the following holds. Let  $v\in B^{- \gamma }_{1,1}$. 
\begin{itemize}
\item[A.] There  is   $\alpha \in B^{- \gamma }_{1,1}$ that satisfies 
$$v= \alpha\circ F - g^{\beta} \alpha$$ 
if and only if 
$$\langle \frac{v}{g^\beta}, \psi\rangle =0$$
for every $\psi \in E_1(L_{1+\beta})$. Moreover the solution is unique if and only if $1\not\in sp(L_{1+\beta})$. 
\item[B.] Consider the finite dimensional subspace 
$$W_\beta  =  \bigoplus_{|\lambda|\geq 1} E_\lambda(L_{1+\beta}).$$
Suppose 
$$\langle \frac{v}{g^\beta}, \psi\rangle =0$$
for every $\psi \in W_\beta$.  Then a solution $\alpha$   is given by 
\begin{equation} \langle \alpha, \psi\rangle  =- \sum_{n=1}^\infty \langle \frac{v \circ F^{n-1}}{ g_n^\beta} , \psi \rangle. \label{soldist} \end{equation} 
\item[C.]  For $\hat{\beta}=0$ and $\beta \sim 0$ we have that there is an unique $\lambda_{1+\beta} \in sp(L_{1+\beta})$ with $|\lambda_{1+\beta}|=r(L_{1+\beta})$. Moreover $\lambda_{1+\beta} \sim 1$ is a simple  isolated eigenvalue of $L_{1+\beta}$ such that $\beta\mapsto   \lambda_{1+\beta} $ is complex analytic and we can choose a $\lambda_{1+\beta}$-eigenvector $\rho_{1+\beta}\in \mathcal{B}^\gamma_{\infty,\infty}$ of $L_{1+\beta}$ such that $$\beta\mapsto  \rho_{1+\beta}$$ is complex-analytic and  $\rho_{1+\beta}> 0$ for $\beta\in \mathbb{R}$, and also a complex analytic  family of  $\lambda_{1+\beta}$-eigenvectors $$\beta\mapsto m_{1+\beta}\in \mathcal{B}^{-\gamma}_{1,1}$$ of $T$. If 
$$\langle \frac{v}{g^\beta},  \rho_{1+\beta}\rangle =0$$
then $\alpha$ given by (\ref{soldist}) is a solution.  Moreover  if $\beta\neq 0$ and  $\beta\sim 0$  this is the only solution  in $\mathcal{B}^{-\gamma}_{1,1}$.
  \end{itemize} 
\end{mainthm}

The case $\beta=0$ is of special interest for us.

\begin{mainthm}[Distributional solutions for the Lvisic equation I]  \label{erer}  Let $0< s<  \gamma$.  Let $\log g \in \mathcal{B}^\gamma_{\infty,\infty}$.  Consider the closed subspace
$$V_0=\{ v\in \mathcal{B}_{1,1}^{-s}\colon \langle v, \rho\rangle =0\}.$$
For every $\phi \in V_0$  there is an unique  distribution  (up to addition by a multiple of $1$)  $$\psi =\psi_\phi \in \mathcal{B}^{-s}_{1,1}$$  that is the solution for the  Liv\v{s}ic  cohomological equation 
$$\phi= \psi\circ F - \psi.$$
Indeed, up to the addition of  a constant, $\psi$ is equal to 
$$ \psi_\phi  = -\sum_{k=0}^\infty  \phi\circ F^k.$$
Moreover   $$\phi\in V_0 \mapsto \psi_\phi \in  V_0\subset \mathcal{B}^{-s}_{1,1} $$ is a bounded linear transformation. 
\end{mainthm} 

\begin{mainthm}[Distributional solutions for the Lvisic equation II]  \label{erer2}  Let $0< s<  \gamma$, with $\gamma\leq 1$. Let $\log g \in \mathcal{B}^\gamma_{\infty,\infty}$.  Consider the closed subspace of $\mathcal{B}^{-s}_{\infty,\infty}$ 
$$W_0=\overline{\{ v\in \mathcal{T} \colon  \int v  \rho \ dm=0\}}.$$
For every $\phi \in W_0$  there is an unique distribution   (up to addition by a multiple of $1$) $$\psi=\psi_\phi\in \mathcal{B}^{-s}_{\infty,\infty}$$  that is the solution for the  Liv\v{s}ic  cohomological equation 
$$\phi=\psi\circ F - \psi.$$
Indeed, up to the addition of  a constant, $\psi$ is equal to 
$$ \psi_\phi  = -\sum_{k=0}^\infty  \phi\circ F^k.$$
Moreover   $\phi\in W_0 \mapsto \psi_\phi  \in W_0 \subset   \mathcal{B}^{-s}_{\infty,\infty} $ is a bounded linear transformation. 
\end{mainthm} 

 Remarkably, the {\it martingale approximations}  of such  distributional solutions of the Lvisic equation satisfy a certain Central Limit Theorem. Given $x\in I$ denote by $P_k(x)$ the unique element of $\mathcal{P}^k$ such that $x \in P_k(x)$. Given a distribution $\psi\in  \mathcal{B}^{-s}_{1,1}$ one can see
$$ \psi\left( \frac{1_{P_n(x)}}{|P_n(x)|}\right)$$
as the average of $\psi$ on the set $P_n(x)$.

\begin{mainthm}[Central limit theorem for martingale approximations]  \label{CLTweak}  Let $\phi \in \mathcal{B}^s_{1,1}$ with $s > 1/2$, be a real valued function  with zero average with respect to $\rho m$, and suppose $\sigma(\phi)\neq 0$. Let $\psi \in \mathcal{B}^{-s}_{1,1}$ be as  in Theorem \ref{erer}. Then 
$$\lim_{n\rightarrow +\infty} m\left(x\in I\colon \ \frac{1}{\sqrt{n}} \psi\left( \frac{1_{P_n(x)}}{|P_n(x)|} \right) \leq z \right)=  \frac{1}{\sigma(\phi) \sqrt{2\pi}} \int_{-\infty}^z e^{-\frac{t^2}{2\sigma^2(\phi)}}  \  dt,$$
that is  
$$\lim_{n\rightarrow +\infty} m(x\in I\colon \ \frac{\sum_{k=0}^{n-1}   c_0(\psi,P^k(x))\phi_{P^k(x)}(x)  }{\sqrt{n}} \leq z ) =  \frac{1}{\sigma(\phi)  \sqrt{2\pi}} \int_{-\infty}^z e^{-\frac{t^2}{2\sigma^2(\phi)}}  \  dt.$$

The same holds for $\phi \in \mathcal{B}^s_{\infty,\infty}$ and in this case additionally  there is $C$ such that 
$$\left| \psi\left( \frac{1_{P_n(x)}}{|P_n(x)|}\right)    -\sum_{k< n} \phi(F^k(x)) \right|\leq C$$
for every $x\in I$.
\end{mainthm}

\subsection{The connection between the Liv\v{s}ic and twisted equations: fractional derivatives} 

The solutions of the twisted equation and of the Liv\v{s}ic equation are connected through {\it fractional derivatives.}  Given a distribution

$$\psi =c_01_I + \sum_{k=0}^\infty \sum_{P\in \mathcal{P}^k} c_P \phi_P,$$

we define the $\beta$-fractional derivative of $\psi$ as the distribution 

$$D^\beta \psi = \sum_{k=0}^\infty \sum_{P\in \mathcal{P}^k} c_P |P|^{-\beta}  \phi_P$$

Here $\beta\in \mathbb{C}$ (in particular  $\psi - D^0\psi$ is a constant function).  Note that even if $\phi$ is a  function   then it is not necessarily true that the r.h.s. converges as a function, so one need to see the r.h.s. as either a formal series or as a distribution.  

Notice that if $s \in \mathbb{R}$ then
$$D^\beta\colon \mathcal{B}^s_{1,1}\rightarrow \mathcal{B}^{s-\operatorname{Re} \beta}_{1,1}$$
and
$$D^\beta\colon \mathcal{B}^s_{\infty,\infty}\rightarrow \mathcal{B}^{s-\operatorname{Re} \beta}_{\infty,\infty}$$
are  continuous and onto linear maps whose kernel is the space of constant functions.

\begin{mainthm}[Solutions with low regularity, Liv\v{s}ic equation  and Birkhoff sums] \label{azaz}  Let $0< s< \operatorname{Re}  \beta<   \gamma$. Suppose that $v\in \mathcal{B}^\gamma_{1,1}$ (respectively   $v\in  \mathcal{B}^\gamma_{\infty,\infty}$) and  $\log g \in \mathcal{B}^\gamma_{\infty,\infty}$. Let $\alpha_v \in L^1(m)$ be a solution 
for the cohomological equation
\begin{equation}\label{ch2} v= \alpha_v\circ F-  g^{\beta} \alpha_v. \end{equation} 
Then $\alpha_v \in  \mathcal{B}^{s}_{1,1}$, that is, $D^\beta \alpha_v$  is a  {\bf distribution} in  $\mathcal{B}^{s-\operatorname{Re}  \beta}_{1,1}$. Moreover 
$$\phi_v := (D^\beta \alpha_v)\circ F - D^\beta \alpha_v$$
is a {\bf function} in $\mathcal{B}_{1,1}^{ \gamma- \operatorname{Re} \beta  -\epsilon }$ (respectively   $\phi_v\in  \mathcal{B}^{\gamma- \operatorname{Re} \beta -\epsilon}_{\infty,\infty}$)  for every small $\epsilon > 0$, and 
$$\int \phi_v \ d\mu =0.$$ 
Moreover, up to an additive constant,
$$D^\beta \alpha_v  = - \sum_{k=0}^\infty  \phi_v\circ F^k.$$
and for every  $\delta  \in (0,\gamma)$ we have that $\psi_v= D^\beta \alpha_v$ is the unique  distribution in $\mathcal{B}^{-\delta}_{1,1}$ (respectively $\mathcal{B}^{-\delta}_{\infty,\infty}$  )  that is a solution for Liv\v{s}ic cohomological equation
$$\phi_v= \psi_v\circ F -\psi_v.$$
\end{mainthm} 

We say that a continuous function $\alpha$ has $O$-modulus of continuity $h\colon \mathbb{R}_{> 0} \rightarrow \mathbb{R}_{> 0}$, with $\lim_{\delta\mapsto 0^+} h(\delta)=0$  at $x_0\in I$ if there is $C\geq 0$ such that 

$$|\alpha(x)-\alpha(x_0)|\leq C h(d(x,x_0)),$$
and it has {\it small  $o$-modulus of continuity} $h$ at $x_0$  if 

$$\lim_{x\rightarrow x_0}  \frac{|\alpha(x)-\alpha(x_0)|}{h(d(x,x_0))} =0.$$

\begin{mainthm}[Dichotomy and Liv\v{s}ic equation I] \label{dic11} Let $0<   \beta  < \gamma$.   Suppose that $v\in   \mathcal{B}^\gamma_{\infty,\infty}$ is a real-valued function  and  $\log g \in \mathcal{B}^\gamma_{\infty,\infty}$  Let $\alpha_v,\phi_v $  be  as in Theorem \ref{azaz}. 
The following statements are equivalent
\begin{itemize}
\item[I.] We have 
$\alpha_v \in   \mathcal{B}^{\gamma}_{\infty,\infty}$. 
\item[II.] We have 
$\alpha_v \in  \mathcal{B}^{\operatorname{Re} \beta +\epsilon}_{\infty,\infty}$ for some $\epsilon> 0$.
\item[III.] We have that $\alpha_v$ has small o-modulus of continuity $\delta^\beta$ at a  point  $x_0\in I$.
\item[IV.] We have 
$$\liminf_{n\rightarrow +\infty} \min_{P\in \mathcal{P}^n} \frac{1}{|P|^{\beta} }   \max_{\substack{x,y\in P}}|\alpha_v(y)-\alpha_v(x)|=0.$$
\item[V.] For $s> 0$ let 
\begin{equation}\label{coefa}   c_s(\alpha_v,P):= \frac{1}{|P|^s}  \int \alpha_v \phi_P \ dm.\end{equation} 
Then 
$$\lim_{n\rightarrow +\infty} \min_{P\in \mathcal{P}^n} |c_{\beta} (\alpha_v,P)| =0.$$
\item[VI.] There is $\psi\in L^1(m)$ such that 
\begin{equation} \label{liv44} \phi_v =  \psi\circ F - \psi.\end{equation} 
\item[VII.] We have   $\sigma^2(\phi_v)=0$.
\end{itemize}
\end{mainthm}

Todorov~\cite[Theorem~1]{todorov} obtained a dichotomy result for solutions of the twisted cohomological equation in a more regular setting. 
In contrast, Theorem~\ref{dic11} yields sharper conclusions and, more importantly, establishes a new connection between this dichotomy, the Livšic cohomological equation, and the asymptotic variance that will be central to obtain our main  results.  Furthermore, our methods allow us to treat observables $v$ with far lower regularity, as described in the theorem below.

\begin{figure}[h]
\includegraphics[scale=0.34]{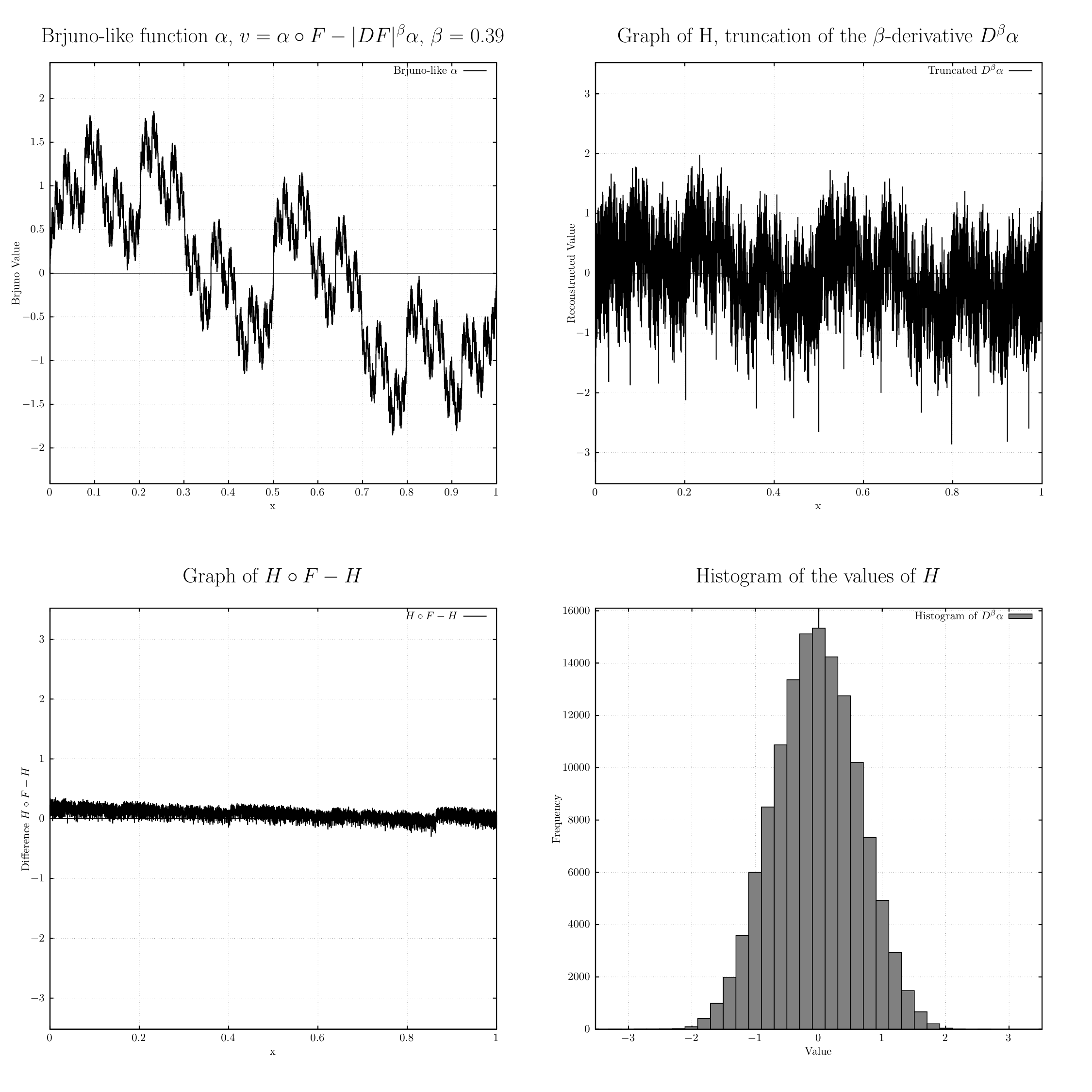}
\caption{}
\end{figure}

\begin{remark}
It is fairly easy to implement unbalanced Haar wavelets and the associated fractional derivatives on a modern computer. We implemented the algorithms in C++ to illustrate some of the results of this work. In Figure 1 we consider the expanding map on the circle
$$F(x)= 2x + 0.1\sin(2\pi x) \mod 1 $$
with $v(x)= \sin(2\pi x)$ and $\beta=0.39$. 

In this case it appears that $\sigma(\phi_v) > 0$, where the function $\phi_v$ is as in Theorem \ref{azaz}, since the solution $\alpha=\alpha_v$ behaves quite wildly. We have that  $D^\beta \alpha$ is a distribution and is (up to a constant) equal to the infinite Birkhoff sum
$$
-\sum_k \phi_v \circ F^k .
$$

Nevertheless, one can compute numerically $\alpha$, its Haar decomposition, and consequently the Haar decomposition of its $\beta$-fractional derivative $D^\beta \alpha$, up to a finite level $n$ ($n = 17$ in this figure). This defines a {\it function} $H$, whose values, as predicted by Theorem \ref{CLTweak} for the case $\sigma(\phi_v)> 0$, exhibit a probabilistic distribution close to a normal distribution (see Figure 1). The rate of convergence is not fast, likely of order $1/\sqrt{n}$.

The same theorem tells us that (up to addition by a constant)
$$
H(x)\sim -  \sum_{k=0}^{n-1}\phi_v(F^k(x)),
$$
and in particular the local oscillations of
$$
H(F(x)) - H(x) = -\phi_v(F^n(x)) + \phi_v(x)
$$
must be much smaller than the local oscillations of $H(x)$, which is indeed observed in the experiment.
\end{remark}

\begin{mainthm}[Dichotomy and Liv\v{s}ic equation II] \label{dic1} Let $0< \operatorname{Re}  \beta  < \gamma$. Suppose that $v\in \mathcal{B}^\gamma_{1,1}$, with $\gamma -\operatorname{Re} \beta > 1/2$, is a real-valued function  and  $\log g \in \mathcal{B}^\gamma_{\infty,\infty}$.  Let $\alpha_v,\phi_v $  be  as in Theorem \ref{azaz}. 
The following statements are equivalent
\begin{itemize}
\item[I.] We have 
$\alpha_v \in \mathcal{B}^{\gamma -\delta}_{1,1}$ for every $\delta > 0$. 
\item[II.] We have 
$\alpha_v \in \mathcal{B}^{\beta+\delta}_{1,1}$, for  some $\delta > 0$.
\item[III.] There is a set of points $S$ with positive measure such that if $x\in S$ we have  
$$\lim_n   \frac{1}{\sqrt{n} }\sum_{j\leq n }    \frac{\phi_{P^j(x)}(x)}{|P^j(x)|^{\beta}} \int \alpha_v(y) \phi_{P^j(x)}(y) dm(y)  =0.$$
\item[IV.] There is $\psi\in L^2(m)$ such that 
$$\phi_v =  \psi\circ F - \psi.$$
\item[V.] We have   $\sigma^2(\phi_v)=0$.
\end{itemize}
\end{mainthm}

\subsection{Dichotomies in real analytic families and Banach spaces of functions}

We are ready to state the main general  dichotomies results. The main innovation here is that $F$ is not required to be piecewise linear, and furthermore the functions $v_t\in \mathcal{B}^\gamma_{1,1}$ need not be continuous, or even bounded.

\begin{mainthm}[Dichotomy for analytic families]\label{dic2} Let $0< \beta  < \gamma$.  Let $J$ be an open interval and $t\in J\mapsto v_t$ be a real analytic family of real-valued functions $v_t\in  \mathcal{B}^\gamma_{\infty,\infty}$ (respec. $v_t\in  \mathcal{B}^\gamma_{1,1}$ with   $\gamma -\beta > 1/2$). Suppose that $\log g \in \mathcal{B}^\gamma_{\infty,\infty}$, with $s,\beta \in (0,\gamma)$.  Then either for all $t$ the solution $\alpha_t$ of  
$(\ref{ch2})$ satisfies $I$-$VII$ in Theorem \ref{dic11}  ( respec.  $I$-$V$ in Theorem  \ref{dic1}) or this holds only for a countable number of isolated parameters. 
\end{mainthm} 

\begin{mainthm}[Dichotomy in Banach Spaces of Functions]\label{dic4} Let $0< \beta  < \gamma$.  Let $\mathcal{B}$ be a Banach space of  functions continously embedded in  $\mathcal{B}^\gamma_{\infty,\infty}$ (resp. $\mathcal{B}^\gamma_{1,1}$).  Suppose there is $w\in \mathcal{B}$ such that  the solution $\alpha_w$ does not satisfy  $I$-$VII$ in Theorem \ref{dic11}  ( respec.  $I$-$V$ in Theorem  \ref{dic1}). Then 
\begin{itemize} 
\item[A.]  For a residual set of  functions $$t\in [0,1]^d\mapsto v_t$$  in  $C^k([0,1]^d,\mathcal{B})$,  with $k\in \mathbb{N}$ and $d\geq 1$, the set  of all $t\in J$ such that the solution $\alpha_t$ of  
$(\ref{ch2})$  satisfy $I$-$V$ in Theorem \ref{dic11}  ( respec.  $I$-$V$ in Theorem  \ref{dic1})  has zero Lebesgue measure.
\item[B.]  The set of functions in $\mathcal{B}$ that satisfy   $I$-$V$ in Theorem \ref{dic11}  ( respec.  $I$-$V$ in Theorem  \ref{dic1}) is a closed set with empty interior.
\end{itemize} 

Moreover  if $\log g \in \mathcal{B}^{\gamma+\delta}_{\infty,\infty}$, with $\delta > 0$,  then   the assumptions of this theorem are satisfied by $\mathcal{B}^\gamma_{\infty,\infty}$ and  $\mathcal{B}^\gamma_{1,1}$.
\end{mainthm} 

\begin{mainthm}[Dichotomy with respect to $\beta$] \label{dic3}   Let $0< \beta  < \gamma$.  Given  a real analytic function $\beta\in J=(0,1)\mapsto v_\beta \in  \mathcal{B}^\gamma_{\infty,\infty}$ (respec. $\beta \in J=(0, \gamma-1/2)\mapsto v_\beta \in  \mathcal{B}^\gamma_{1,1}$).   Then either for all $\beta\in J$ the solution $\alpha_\beta$ of  
$(\ref{ch2})$ belongs to $\mathcal{B}^{\gamma}_{\infty,\infty}$ ( respec. $\mathcal{B}^{\gamma -\delta}_{1,1}$, for all small $\delta> 0$) or this holds only for a countable  number of isolated parameters $\beta$. Moreover for $v_\beta \in  \mathcal{B}^\gamma_{\infty,\infty}$ those parameters do not accumulate to $0$.
\end{mainthm}

\section{Chain and Leibniz rules for fractional derivatives} 
When $g$ is constant, that  is the "linear" case  (for instance $F(x)=2x \ mod 1$), dealing with fractional  derivatives is fairly easy.   The key to deal with  non-linear maps $F$  is to see how fractional derivatives behaves with respect to compositions (Chain  Rule) and products (Leibniz Rule) of functions.   Given $r\in \mathbb{R}$ let $r_+=\max\{0,r\}$ and $r_-=\min\{0,r\}$. Moreover let $r_0=1$ if $r=0$, and $r_0=0$ otherwise. 

\begin{theorem}[Chain Rule I]\label{chain2} Suppose that $\log g 
\in \mathcal{B}^{\gamma}_{\infty,\infty}$, with $\gamma > 0$. Let $s, \operatorname{Re} \beta\in  \mathbb{R}$, with 
$$ s-\operatorname{Re}  \beta+\gamma> 0 \text{ and } s+\gamma > 0,$$   and a small $\epsilon > 0$   there is a bounded linear transformation
$$R_{C}\colon \mathcal{B}^s_{\infty,\infty} \rightarrow \mathcal{B}^{\gamma+(s-\operatorname{Re} \beta)_--\epsilon}_{\infty,\infty} $$ such that the following holds.  If $\psi \in \mathcal{B}^s_{\infty,\infty}$  then
$$D^\beta (\psi\circ F)= (D^\beta \psi)\circ F\cdot g^{\beta} + R_{C}(\psi).$$
\end{theorem}

Define $ \hat{\epsilon}  =\epsilon$ if $s- \operatorname{Re} \beta \geq  0$ and  $ \hat{\epsilon} = 0$ otherwise.

\begin{theorem}[Chain Rule II]\label{chain} Suppose that $\log g 
\in \mathcal{B}^{\gamma}_{\infty,\infty}$, with $\gamma > 0$. Given $s, \operatorname{Re} \beta\in  \mathbb{R}$, with 
$$ s-\operatorname{Re}  \beta+\gamma> 0 \text{ and } s+\gamma > 0,$$ and a small $\epsilon > 0$, define $ \hat{\epsilon}  =\epsilon$ if $s- \operatorname{Re} \beta \geq  0$ and  $ \hat{\epsilon} = 0$ otherwise.   There is a bounded linear transformation
$$R_{C}\colon \mathcal{B}^s_{1,1} \rightarrow \mathcal{B}^{\gamma+(s-\operatorname{Re} \beta)_-- \hat{\epsilon}}_{1,1}$$ such that the following holds.  If $\psi \in \mathcal{B}^s_{1,1}$  then
$$D^\beta (\psi\circ F)= (D^\beta \psi)\circ F\cdot g^{\beta} + R_{C}(\psi).$$
\end{theorem}

 Define $ \tilde{\epsilon}  =\epsilon$ if $s \geq  0$ and  $ \tilde{\epsilon} = 0$ otherwise.

\begin{theorem}[Leibniz Rule I]\label{leibnitz} Suppose that  $h \in \mathcal{B}^{\gamma}_{\infty,\infty}$.  Given $ s, \operatorname{Re} \beta\in \mathbb{R}$ and small $\epsilon > 0$,  such that 
$$ s-\operatorname{Re}  \beta+\gamma> 0, \gamma +s > 0.$$
  There    is a  bounded linear transformation
$$R_{L}\colon \mathcal{B}^s_{1,1} \rightarrow \mathcal{B}^{t}_{1,1},$$
where
$$t=\min \{\gamma+ (s - \operatorname{Re} \beta)_- + \hat{\epsilon}, \gamma  + s_-  - \operatorname{Re}\beta  - \tilde{\epsilon} \}$$
 such that the following holds.  If $\psi \in \mathcal{B}^s_{1,1}$  then
$$D^\beta (\psi h)= (D^\beta \psi) \cdot h + R_{L}(\psi).$$
\end{theorem} 

 Define $ \overline{\epsilon}  =\epsilon$ if $s=0$.
 
\begin{theorem}[Leibniz Rule II]\label{leibnitz1} Suppose that $h \in \mathcal{B}^{\gamma}_{\infty,\infty}$. Given $s,\operatorname{Re} \beta\in \mathbb{R}$   and small $\epsilon > 0$  such that 
$$ s-\operatorname{Re}  \beta+\gamma> 0, \gamma +s > 0,$$
there is a  bounded linear transformation
$$R_{L}\colon \mathcal{B}^s_{\infty,\infty} \rightarrow \mathcal{B}^{t}_{\infty,\infty}$$
where
$$t = \min\{ (s - \beta)_- - \epsilon + \gamma,  s_-+ \gamma - \beta -\overline{\epsilon} \}$$
 such that the following holds.  If $\psi \in \mathcal{B}^s_{\infty,\infty}$  then
$$D^\beta (\psi h)= (D^\beta \psi) \cdot h + R_{L}(\psi).$$
\end{theorem} 

\begin{remark} A crucial point on the Chain and Leibniz rules  above is that while
$$D^\beta (\psi\circ F)  \text{ and } (D^\beta \psi)\circ F\cdot g^{\beta} $$
may not be even functions, their difference is indeed a (regular) function nevertheless, and the same holds to   
$$D^\beta (\psi  g^\beta) \text{ and }  (D^\beta \psi)  g^{\beta}.$$
\end{remark}

\section{Solutions with low regularity for H\"older data}

\begin{proof}[Proof of Theorem   \ref{main4}.I] The following argument is quite similar to that in 
Baladi and S. \cite{alt}. Let $0< s< \beta_0\leq Re \ \beta < \gamma$ and $|Im \ \beta|< \delta$.  Let $\alpha_0=0$ and 
$$\alpha_{n+1} = \frac{v-\alpha_n\circ F}{g^{\beta}}.$$
It is easy to see that
$$\sup |\alpha_n|_{L^\infty(m)}\leq  C_1(\beta_0) < \infty$$
and 
$$\lim_n \alpha_n(x)=\alpha(x)$$
for every $x\in I$. 
Choose $\epsilon > 0$ small such that 
$$\lambda= \sup_{x\in I}  \frac{(1+\epsilon)^s}{g^{ \beta_0 -s} (x)}< 1$$
and  $\delta > 0$ small such that
$$  \sup_{0<d(x,y)< \delta} \frac{d(F(x),F(y))}{g(x) d(x,y)}<   1+\epsilon.$$

Let $C_0=0$ and let 
$$C_n=\sup_{ x\neq y } \frac{|\alpha_n(x)-\alpha_n(y)|}{d(x,y)^s}.$$
 It is enough to show  $\sup_n C_n < \infty$.
Considering the cases $d(x,y)< \delta$ and $d(x,y) \geq \delta$ one can show that 
\begin{align*}  & |\alpha_{n+1}(x)-\alpha_{n+1}(y)|  \\
&\leq    \Big| \frac{v(x)}{g^{\beta}(x)}- \frac{v(y)}{g^{\beta}(y)}\Big| + \frac{|\alpha_n(f(x))-\alpha_n(f(y))|}{g^{\operatorname{Re} \beta}(x)} + |\alpha_n(y)|\Big| \frac{1}{g^{\beta}(x)}- \frac{1}{g^{\beta}(y)}   \Big|\\
& \leq ( C_2(\beta_0,\delta) |v|_{\mathcal{B}^s_{\infty,\infty}} + \lambda  C_n) d(x,y)^s,\end{align*} 
so $C_{n+1}\leq  C_2(\beta_0,\delta) |v|_{\mathcal{B}^s_{\infty,\infty}}  + \lambda C_n$. Since $\lambda \in (0,1)$ it follows that $$\sup_n C_n \leq C_3(\beta_0,\delta) |v|_{\mathcal{B}^s_{\infty,\infty}}  < \infty.$$
\end{proof}

\begin{proof}[Proof of Theorem \ref{main4}.II and \ref{main4}.III] Let $\operatorname{Re} \beta\leq s$, $\operatorname{Re} \beta <   \gamma$ and assume $v\in \mathcal{B}^s_{\infty,\infty}$, and $g^{\beta} \in \mathcal{B}^\gamma_{\infty,\infty}$.  There is $C> 0$ such that for every $x,y \in I$, $x\neq y$ there is $N=N(x,y)$ such that 
\begin{equation}\label{ty1}  \frac{1}{C} d(x,y)\leq    \frac{1}{\Pi_{j=0}^{N-1} g\circ F^j(x)}\leq C d(x,y),\end{equation} 
\begin{equation}\label{ty2}   \frac{1}{C} d(x,y)\leq    \frac{1}{\Pi_{j=0}^{N-1} g\circ F^j(y)}\leq C d(x,y).\end{equation} 
Note that 
\begin{align*} \alpha(y) -\alpha(x) &= \sum_{k=1}^{N-1}  \frac{-v\circ F^k(x)}{\Pi_{j=0}^{k-1} g^{\beta}\circ F^j(x)} - \sum_{k=1}^{N-1}  \frac{-v\circ F^k(y)}{\Pi_{j=0}^{k-1} g^{\beta}\circ F^j(y)} \\ 
&+  \sum_{k=N}^{\infty}  \frac{-v\circ F^k(y)}{\Pi_{j=0}^{k-1} g^{\beta}\circ F^j(y)} - \sum_{k=N}^{\infty}  \frac{-v\circ F^k(x)}{\Pi_{j=0}^{k-1} g^{\beta}\circ F^j(x)}.\end{align*} 
We have
$$\Big|\sum_{k=N}^{\infty}  \frac{-v\circ F^k(y)}{\Pi_{j=0}^{k-1} g^{\beta}\circ F^j(y)} - \sum_{k=N}^{\infty}  \frac{-v\circ F^k(x)}{\Pi_{j=0}^{k-1} g^{\beta}\circ F^j(x)}   \Big|\leq   C d(x,y)^{\operatorname{Re} \beta}.  $$
Furthermore  
\begin{align*} |\alpha(y) -\alpha(x)| &\leq  \sum_{k=1}^{N-1}   \frac{|v\circ F^k(y) - v\circ F^k(x)|}{\Pi_{j=0}^{k-1} g^{\operatorname{Re} \beta}\circ F^j(y)} \\
&+ \sum_{k=1}^{N-1} \frac{|v\circ F^k(x)|}{\Pi_{j=0}^{k-1} g^{\beta}\circ F^j(x)}  \Big|  \frac{\Pi_{j=0}^{k-1} g^{\beta}\circ F^j(x)}{\Pi_{j=0}^{k-1} g^{  \beta}\circ F^j(y)} -  1 \Big|\\
&\leq C d(x,y)^\beta \sum_{k=1}^{N-1} \Big(\Pi_{j=0}^{k-1} g^{s-\operatorname{Re} \beta}\circ F^j(y)  \Big) d(x,y)^{s-\operatorname{Re} \beta} + \\
&+ C \sum_{k=1}^{N-1}  \frac{d(F^ky,F^kx)^\gamma}{\Pi_{j=0}^{k-1} g^{\operatorname{Re} \beta}\circ F^j(x)}  \end{align*} 
If $s > \operatorname{Re} \beta$ we have that there is $C$ such that 
$$\sum_{k=1}^{N-1} \Big(\Pi_{j=0}^{k-1} g^{s-\operatorname{Re} \beta}\circ F^j(y)  \Big) d(x,y)^{s-\operatorname{Re} \beta} \leq C,$$
and if $\operatorname{Re} \beta=s$ we have 
$$\sum_{k=1}^{N-1} \Big(\Pi_{j=0}^{k-1} g^{s-\operatorname{Re} \beta}\circ F^j(y)  \Big) d(x,y)^{s-\operatorname{Re} \beta}=N \leq C (|\ln d(x,y)|+1).$$
We also have
\begin{align*} \sum_{k=1}^{N-1}  \frac{d(F^ky,F^kx)^\gamma}{\Pi_{j=0}^{k-1} g^{\operatorname{Re} \beta}\circ F^j(x)}&\leq \sum_{k=1}^{N-1}  \frac{ d(F^ky,F^kx)^{\gamma-\operatorname{Re} \beta}  d(F^ky,F^kx)^{\operatorname{Re} \beta}}{\Pi_{j=0}^{k-1} g^{\operatorname{Re} \beta}\circ F^j(x)} \\
&\leq C \sum_{k=1}^{N-1}  \frac{ d(F^ky,F^kx)^{\gamma-\operatorname{Re} \beta}}{\Pi_{j=0}^{N-1} g^{\operatorname{Re} \beta}\circ F^j(x)} \\
&\leq C d(y,x)^{\operatorname{Re} \beta} \sum_{k=1}^{N-1} d(F^ky,F^kx)^{\gamma-\operatorname{Re} \beta} \leq
C d(y,x)^{\operatorname{Re} \beta}. \end{align*} 
\end{proof}

\section{Birkhoff sum as a  distributional solution  of  the Lvisic equation}

\begin{proof}[Proof of Theorem \ref{erer}] Since $s\in (0,\gamma)$,  the  transfer operator $\Phi_g$ is a bounded operator acting on $\mathcal{B}^{s}_{\infty,\infty}$, it has spectral radius one  and spectral gap, so there is $C$ and $\theta \in (0,1)$ such that 
$$\left|\Phi_g^k u - \int u \ dm\ \rho  \right|_{\mathcal{B}^{s}_{\infty,\infty}} \leq C\theta^k |u|_{\mathcal{B}^{s}_{\infty,\infty}} $$
for every function $u \in \mathcal{B}^{s}_{\infty,\infty}$.  

Consider the closed subspace
$$V_0=\{ v\in \mathcal{B}_{1,1}^{-s}\colon v(\rho)=0\}.$$

We claim  that $V_0\cap \mathcal{T}$ is dense in $S$. Indeed if  $v\in V_0$ there is a sequence $v_n \in \mathcal{T}$ such that $v_n\rightarrow_n v$ on  $\mathcal{B}_{1,1}^{-s}$.  Since $v(\rho)=0$ we have 

$$\lim_n \int v_n \rho \ dm=0,$$
So 

$$w_n=v_n -\int v_n \rho \ dm 1_I\rightarrow_n v$$
on $\mathcal{B}_{1,1}^{-s}$ and
$$\int w_n \rho \ dm =0.$$
Given $w\in V_0\cap \mathcal{T}$ consider  the formal sum
$$ \psi_w = - \sum_{k=0}^\infty w\circ F^k.$$
We claim that $\psi_w$ defines a distribution that belongs to $\mathcal{B}_{1,1}^{-s}$. Indeed due the spectral gap 
\begin{align*} |\langle \psi_w,  \theta \rangle | &= \left| - \sum_{k=0}^\infty \int w\circ F^k  \theta  \ dm  \right|   \leq \sum_{k=0}^\infty \left|\int w    \Phi_g^k \theta  \ dm   \right|  \\
 &\leq \sum_{k=0}^\infty \left|\int w   \left(  \Phi_g^k  \theta  -\int  \theta dm \ \rho  \right)dm   \right| 
\leq C |w|_{\mathcal{B}^{-s}_{1,1}} |\theta|_{\mathcal{B}^{s}_{\infty,\infty}} \end{align*} 
so $\psi_w$ defines a distribution and 
$$ |\psi_w|_{\mathcal{B}^{-s}_{\infty,\infty}}\leq   C |w|_{\mathcal{B}^{-s}_{1,1}}.$$
So $w\mapsto \psi_w$ extends to a  bounded linear transformation  acting on $V_0$.  It is easy to see that $\psi_w$ is the unique solution of the cohomological equation. 
\end{proof} 

\begin{proof}[Proof of Theorem \ref{erer2}]  The proof is quite similar to the proof of Theorem \ref{erer}, but since $\mathcal{B}^{-s}_{\infty,\infty}$ is not separable we need to be little more careful. Due S. \cite{as2} (see also  Nakano and Sakamoto \cite{ns}) we have that $\Phi_g$ has spectral gap acting on  $\mathcal{B}^s_{1,1}$ . Since $(\mathcal{B}^s_{1,1})^\star= \mathcal{B}^{-s}_{\infty,\infty}$ we can define  $\psi_w$ for $w\in V_0\cap \mathcal{T}$ as before and show that $\psi_w\in V_0$, and $w\mapsto \psi_w$ extends to a bounded operator in $V_0$.
\end{proof}

\section{Solutions with low regularity  and Birkhoff sums} 

\begin{proof}[Proof of Theorem \ref{azaz}] We assume $v\in \mathcal{B}^\gamma_{1,1}$ (the other case is similar).  Define
$$\alpha =   \sum_{k=0}^\infty  \frac{v\circ F^k}{\Pi_{j=0}^{k-1} g^{\beta} \circ F^j}.$$

We have $v\in L^1$  so  this series converges in $L^1$ since 

 $$\left| \frac{v\circ F^k}{\Pi_{j=0}^{k-1} g^{\beta} \circ F^j}\right|_{L^1} \leq C \frac{|v|_{L^1}}{ (\min g)^{\beta k}}.$$
 
One can see that  $v=\alpha\circ F - g^\beta \alpha$. In particular $\alpha\in B^{-\delta}_{1,1}$, for every $\delta > 0$.  Choose
$\epsilon, \delta > 0$ small enough  such that 

$$\gamma -\delta -\beta -\epsilon > 0. $$

By Theorems \ref{chain} and  \ref{leibnitz}  we can take the  $\beta$-pseudo derivative of (\ref{ch2}) to obtain
$$D^\beta v = (D^\beta\alpha)\circ Fg^{\beta} +R_C(\alpha) - g^{\beta} (D^\beta \alpha)- R_{L}(\alpha),  $$
so
$$\frac{D^\beta v -R_C(\alpha) +R_{L}(\alpha)  }{g^{\beta}} = (D^\beta\alpha)\circ F- D^\beta\alpha,$$
where 
$$\phi = \frac{D^\beta v -R_C(\alpha) +R_{L}(\alpha) }{g^{\beta}}$$
satisfies
\begin{align*} 
&|\phi|_{\mathcal{B}^{s-\beta}_{1,1}}\leq  |\phi|_{\mathcal{B}^{\gamma- \delta - \beta -\epsilon}_{1,1}}  \\
&\leq C|D^\beta v|_{B^{\gamma-\beta}_{1,1}} 
+ C \left| \frac{-R_C(\alpha) + R_{L}(\alpha)}{g^{\beta}} \right|_{B^{s-\beta}_{1,1}} \\
&\leq C|v|_{B^{\gamma}_{1,1}} 
+ C \left| \frac{-R_C(\alpha) + R_{L}(\alpha)}{g^{\beta}} \right|_{B^{\gamma- \delta - \beta -\epsilon}_{1,1}}  \\
&\leq C|v|_{B^{\gamma}_{1,1}} 
+ C \left| \frac{-R_C(\alpha) + R_{L}(\alpha)}{g^{\beta}} \right|_{B^{\gamma-\delta - \beta -\epsilon}_{1,1}}.
\end{align*} 
In particular by Theorem \ref{erer} we have $D^\beta\alpha \in B^{s-\beta}_{1,1}$ and consequently $\alpha \in B^s_{1,1}.$

\end{proof}

\section{Dichotomy and Liv\v{s}ic Equation: Theorems \ref{dic11} and \ref{dic1}}

\begin{lemma}[Marra and S.\cite{marsma}]\label{ite3}  For $0< s< \gamma$ there is $C>0$ such that for every $n$ and $P\in \mathcal{P}^n$ we have 
$$\left|\Phi_g^k\left(\frac{1_P}{|P|}\right) - \frac{1_{F^kP}}{|F^kP|} \right|_{\scriptscriptstyle  \mathcal{B}^s_{\infty,\infty}}\leq C |F^k(P)|^{\gamma-s}$$
for $k\leq n$.
\end{lemma} 

\subsection{Proof of Theorem \ref{dic11}}  The implications  $$I\implies II\implies III\implies IV \implies V$$ are obvious. 
 
 \begin{proof}[  $V \implies I$]  Due Lemma \ref{ite3} there is $\Cll{yuyu}$ such that for every $P\in \mathcal{P}\setminus \mathcal{P}^0$ 
\begin{equation} \label{appx} \left|\Phi_g\left(\phi_P\right) -\phi_{F(P)}\right|_{\scriptscriptstyle  \mathcal{B}^{Re \ \beta}_{\infty,\infty}}\leq \Crr{yuyu}  |F(P)|^{\gamma-Re \ \beta}.\end{equation} 
Let $c_\gamma$ as in (\ref{coefa}). Then $\alpha \in \mathcal{B}^\gamma_{\infty,\infty}$ if and only if 

$$\sup_{P\in \mathcal{P}} | c_\gamma(\alpha,P)|< \infty.$$

 By Theorem \ref{main4}.C we have $\alpha\in \mathcal{B}^{Re \ \beta}_{\infty,\infty}$ and  (\ref{appx}) we have 

$$| c_\gamma(\alpha\circ F ,P) - \left(\frac{|F(P)|}{|P|}\right)^\gamma  c_\gamma(\alpha,F(P))|\leq  \Crr{yuyu} |\alpha|_{\scriptscriptstyle  \mathcal{B}^{Re \ \beta}_{\infty,\infty}}.$$
Note also that 
$$| c_\gamma(g^\beta \alpha ,P) - \left( \frac{|F(P)|}{|P|} \right)^{  \beta}   c_\gamma(\alpha,P)|\leq  \Cll{yuyu2} |\alpha|_{L^\infty},$$
so since $v=\alpha\circ F - g^\beta  \alpha$ we have  
$$| c_\gamma(\alpha,F(P))|\leq \left( \frac{|P|}{|F(P)|} \right)^{\gamma- \operatorname{Re}  \beta}  |c_\gamma(\alpha,P)| +  \Cll{other}.$$
Define 
$$\theta = \sup_{P\in \mathcal{P}\setminus \mathcal{P}^0}  \left( \frac{|P|}{|F(P)|} \right)^{\gamma - \operatorname{Re}  \beta} < 1.$$
and

$$\Lambda =   \frac{ \Crr{other}}{1-\theta}.$$ 

Note that  if $L > 0$ we have  $$|c_\gamma(\alpha,F(P))|\geq \Lambda  + L\implies |c_\gamma(\alpha,P)|\geq \Lambda  + \left( \frac{|F(P)|}{|P|} \right)^{\gamma - \operatorname{Re}  \beta} L.$$

We claim if $V.$ holds then   $|c_\gamma(\alpha,Q)|\leq L$ for every $Q$, and in particular $\alpha\in \mathcal{B}^\gamma_{\infty,\infty}$.  Indeed otherwise  there is $Q\in \mathcal{P}^{n_0}$ with $n_0\geq 1$ and $L > 0$ such that  $c_\gamma(\alpha,Q)\geq   \Lambda+L$. For every $n$ choose  $P_n\in \mathcal{P}^n$. Then   there is $Q_{n+n_0}\in \mathcal{P}^{n+n_0}$ such that $Q_{n+n_0}\subset P_n$, $F^n(Q_{n+n_0})=Q$ and

$$|c_\gamma(\alpha,Q_{n+n_0})|\geq  \left( \frac{|Q_{n_0}|}{|Q_{n+n_0}|} \right)^{\gamma - \operatorname{Re}  \beta} L,$$
so

\begin{equation}\label{yyy}  |c_\beta(\alpha,Q_{n+n_0})|\geq  \left( \frac{|Q_{n_0}|}{|Q_{n+n_0}|} \right)^{\gamma - \operatorname{Re}  \beta} |Q_{n+n_0}|^{\gamma- \operatorname{Re}  \beta} L = L|Q_{n_0}|^{\gamma- \operatorname{Re}  \beta}   > 0.\end{equation}  

But $V.$ implies 
$$\liminf_{n} |c_\beta(\alpha,Q_{n+n_0})|=0,$$
that contradicts (\ref{yyy}).   
 \end{proof}
 \begin{proof}[  $I \implies VI$]  If $\alpha\in \mathcal{B}^\gamma_{\infty,\infty}$ then $D^\beta \alpha\in \mathcal{B}^{\gamma- \operatorname{Re}  \beta}_{\infty,\infty}\subset L^1$ is a function and Theorem \ref{azaz} tell us that $\psi = D^\beta \alpha$ satisfies the Liv\v{s}ic cohomological equation.
 \end{proof}

\begin{proof}[  $VI\implies VII$]   By \cite{morris2025} we have 
$$\psi \in \bigcap_{\delta> 0} \mathcal{B}^{\gamma- \operatorname{Re}  \beta -\delta}_{\infty,\infty},$$
By Theorem \ref{azaz}  we have 
$$\psi = -\sum_k \phi\circ F^k = D^\beta \alpha$$
in the sense of distributions, so $\phi =\psi\circ F -\phi$, and  in particular  $\sigma^2(\phi)=0$.
\end{proof} 

\begin{proof}[  $VII\implies II$]  If $\sigma^2(\phi)=0$ then by a  classical result for H\"older functions   there is 
$$\psi \in \bigcap_{\delta> 0} \mathcal{B}^{\gamma- \operatorname{Re}  \beta -\delta}_{\infty,\infty} $$
that satisfies the cohomological equation, so Theorem \ref{azaz} tell us that $\psi = D^\beta \alpha$, and consequently $\alpha \in \mathcal{B}^{\gamma  -\delta}_{\infty,\infty}$ for every $\delta > 0$.

\end{proof}

\subsection{Proof of Theorem \ref{dic1}}   Note that $I\implies II$ is easy . 
\begin{proof}[  $II\implies III$]   Let $c_\gamma(\alpha,P)$ as in (\ref{coefa}). We have $\alpha \in \mathcal{B}^{ \operatorname{Re}  \beta+\epsilon}_{1,1}$ if and only if 
$$\| \alpha\|_{\mathcal{B}^{ \operatorname{Re}  \beta+\epsilon}} =\sum_n \sum_{P\in \mathcal{P}^n}  |c_{ \beta+\delta}(\alpha,P)|< \infty,$$
so
$$\sum_n   |c_{\beta}(\alpha,P^n(x))|\leq \sum_n  |P^n(x)|^\delta |c_{\beta+
\delta}(\alpha,P^n(x))|\leq  \sum_n \lambda_2^n  \| \alpha\|_{\mathcal{B}^{\beta+\epsilon}} < \infty,$$
and $III$ holds. 
\end{proof} 

\begin{proof}[  $IV\implies V$]  This is a simple calculation. \end{proof}

\begin{proof}[  $III\implies V$]  Suppose $v\in \mathcal{B}^\gamma_{1,1}$.  By Theorem \ref{azaz} we have that  
$$\phi\in \bigcup_{\epsilon > 0} \mathcal{B}^{\gamma- \operatorname{Re}  \beta -\epsilon}_{1,1}$$
and 
$$\phi = D^\beta \alpha \circ F - D^\beta \alpha.$$
If $\sigma(\phi)\neq 0$  then by  Theorem \ref{CLTweak} and since $c_0(D^\beta\alpha,P^k(x))=c_\beta(\alpha, P^k(x))$ we conclude that  
$$\lim_{n\rightarrow +\infty} m(x\in I\colon \ \frac{\sum_{k=0}^{n-1}   c_\beta(\alpha, P^k(x))\phi_{P^k(x)}(x) }{\sqrt{n}} \leq z )=  \frac{1}{\sigma(\phi)  \sqrt{2\pi}} \int_{-\infty}^z e^{-\frac{t^2}{2\sigma^2(\phi)}}  \  dt,$$
that would contradicts $III$. So $\sigma(\phi)=0$. 
\end{proof}

\begin{proof}[  $V\implies IV \text{ and } 	V\implies I $]  if $\sigma(\phi)=0$ then by Theorem  \ref{varthm}  there is $$\psi\in \mathcal{B}_{1,1}^{ \gamma- \operatorname{Re} \beta  -\epsilon}\subset L^2$$ satisfying $IV$. By Theorem \ref{azaz} we have $D^{ \operatorname{Re}  \beta}\alpha=\psi$, so  $\alpha \in  \mathcal{B}_{1,1}^{ \gamma-  \epsilon}$, for every $\epsilon > 0$, so $I$ holds.  \end{proof} 

\subsection{Proof of Theorem \ref{dic333}} 
\begin{proof}[Proof of Item A] 
The    essential spectrum radius of $L_ {1+\beta}$ acting on $C^\gamma(I)=\mathcal{B}^\gamma_{\infty,\infty}$ is strictly smaller than $1$. One can prove this using methods as in the proof of an analogous result by  Baladi,  Jiang  and Lanford \cite{bjl}   for $C^{1+\beta}$ markovian expanding maps or  using fractional derivatives and  a method similar to Collet and  Isola \cite{ci}. We skip  the details. 

Consider the bounded operator  $T\colon \mathcal{B}^{-\gamma}_{1,1}\rightarrow \mathcal{B}^{-\gamma}_{1,1}$
given by 
$$T(v)= \frac{v\circ F}{g^\beta}.$$
Since $(\mathcal{B}^{-\gamma}_{1,1})^\star= \mathcal{B}^{\gamma}_{\infty,\infty}$ we have  that $T^\star= L_{1+\beta}$.  So either $1\not\in sp(T)$ or it  belongs to the discrete spectrum of both operators and in any case 
$$R(T-Id)= Ker (L_{1+\beta}-Id)_{\perp} = E_1(L_{1+\beta})_\perp.$$
and of course if $\alpha,\hat{\alpha}$ are two solutions of the cohomological equation then $\alpha-\hat{\alpha} \in E_1(T)$ and $1$ is an eigenbvalue of $T$ if and only if $1$ is an eigenvalue of $L_{1+\beta}$. 
\end{proof}

\begin{proof}[Proof of item B.] 
We can decompose $$L_ {1+\beta} = K + R$$
where $K$ and $R$ are bounded operators in $\mathcal{B}^\gamma_{\infty,\infty}$  such that $KR=RK=0$, and there is $\theta < 1$ such that 

$$K(\mathcal{B}^\gamma_{\infty,\infty})=W_\beta; $$
$$sp(K)= sp(L_ {1+\beta})\cap B(0,1)^c;$$
$$sp(R)= sp(L_ {1+\beta})\cap B(0,\theta). $$

For $v\in \mathcal{B}^{-\gamma}_{1,1}$ define   the distribution  $\alpha \in \mathcal{B}^{-\gamma}_{1,1}$  as 

$$\langle \alpha, \psi\rangle  =- \sum_{n=1}^\infty \langle \frac{v \circ F^{n-1}}{ g_n^\beta} , \psi \rangle.$$
Then $\alpha $ is well-defined since for $\psi \in \mathcal{B}^{\gamma}_{\infty,\infty}$ 
\begin{align*} \sum_{n=1}^\infty \left|  \langle \frac{v \circ F^{n-1}}{ g_n^\beta} , \psi \rangle \right|  &= \sum_{n=1}^\infty  \left|  \langle \frac{v}{ g^\beta} , L_{1+\beta}^{n-1}  \left( \psi \right)  \rangle  \right| \\
&= \sum_{n=1}^\infty  \left|  \langle \frac{v}{ g^\beta} , R^{n-1}  \left( \psi \right)  \rangle    \right|  \leq \sum_{n=1}^\infty  C\theta^n < \infty.\end{align*} 
and $v=\alpha \circ F - g^{\beta} \alpha$ in the sense of distributions since 
\begin{align*}
&\left\langle  \alpha  \circ F - g^\beta \alpha , \psi \right\rangle\\
&=
- \sum_{n=1}^{\infty} 
 \langle \frac{v  \circ F^{n}}{ g_n^\beta \circ F  } ,\psi\rangle 
  + \sum_{n=1}^{\infty} \langle  \frac{v  \circ F^{n-1}}{ g_{n-1}^\beta \circ F  } ,\psi\rangle  \\
&= \left\langle v , \psi \right\rangle. \end{align*} 
 \end{proof}

 \begin{proof}[Proof of item C] We have that $\beta\mapsto L_{1+\beta}$ is a complex-analytic family of quasi-compact bounded operators acting on $C^\gamma$. for $\beta=1$ we have that the peripheral spectrum is $\{1\}$  is a simple eigenvalue. This implies that    $\beta\sim 0$ there is a unique $\lambda_{1+\beta}\sim 1$ such that it is a simple eigenvalue and it is   the only element of the peripheral spectrum of $L_{1+\beta}$. Moreover by classical thermodynamical formalism 
 
 $$\frac{d}{d\beta} \ln |\lambda_{1+\beta}| \left.\right|_{\beta=0} = \int \ln g \ \rho \ dm > 0.$$
 
 It follows that $\lambda_{1+\beta}\neq 1$ for $\beta\sim 0$, $\beta \neq 0$ and the remaining statements  follows.

 \end{proof}

\section{Central Limit  Theorem  for distributional solutions: Proof of Theorem \ref{CLTweak}}

\change{"for observables " removed from section's title}

\begin{theorem}[Central Limit Theorem for Besov functions]  \label{varthm}   Suppose that either  $\psi \in \mathcal{B}^t_{1,1}$, $1/2\leq t < \gamma$, or   $\psi \in \mathcal{B}^t_{\infty,\infty}$, and in both cases 
$$\int \psi  \ \rho \ dm=0.$$ We have that  $\sigma^2=\sigma^2(\psi)$, as defined in (\ref{variance}), is well-defined. Then $\sigma=0$ if and only if $\psi = u\circ F - u$, where $u\in \mathcal{B}^t_{1,1}$. Moreover if $\sigma^2(\psi) > 0$ then 
$$\lim_{n\rightarrow +\infty} m(x\in I\colon \ \frac{\sum_{k=0}^{n-1} \psi\circ F^k(x) }{\sqrt{n}} \leq z )=  \frac{1}{\sigma \sqrt{2\pi}} \int_{-\infty}^z e^{-v^2/2\sigma^2}  \  dv.$$
\end{theorem} 
\begin{proof}The case $\psi \in \mathcal{B}^t_{\infty,\infty}$ is a classical one, since functions in  $\mathcal{B}^t_{\infty,\infty}$ are H\"older continuous. Let's consider the other case. The assumption  $1/2\leq t$ implies that $\mathcal{B}^t_{1,1}$ is continously embedded in $L^2(m)=L^2(\rho m)$.  By the exponential decay of correlations proved in Arbieto and S. \cite{as} (See also Nakano and Sakamoto \cite{ns}) it follows that $\sigma$ is well-defined. Consider the transfer operator
$$\Phi_g(\psi)(x)=\sum_{Fy=x} g(y)\psi(y).$$
By Arbieto and S. \cite{as2} we have that  $\Phi_g\colon \mathcal{B}^t_{1,1}\mapsto \mathcal{B}^t_{1,1}$  has spectral gap, and, since $\gamma$-H\"older functions (and in particular $\rho$) are bounded point-wise multipliers acting on $\mathcal{B}^t_{1,1}$ it follows that the 
  normalized transform operator 
  $$T\psi =  \frac{1}{\rho} \Phi (\rho \psi)$$
  also has spectral gap on $\mathcal{B}^t_{1,1}$.  Note that $T$ is a left inverse of the Koopman operator, that is, $T(\psi\circ F)=\psi$.  Define $v\in \mathcal{B}^t_{1,1}$ as
  $$w = \sum_{k=1}^\infty T^k\psi. $$
  It is easy to see that 
  $$h=\psi - (w\circ F - w) \in  Ker \ T,$$
  which implies  $\sigma(\psi)=\sigma(h)=|h|_{L^2(\rho m)}$ and
  $$\int h u\circ F \rho \ dm =0$$ 
  for every $u\in L^2(\rho \ dm)$.   Note that  $\sigma=0$ if and only if $\psi=w\circ F - w$. If $\sigma> 0$ we can apply Gordin \cite{gordin} and obtain the Central Limit Theorem. 
\end{proof}

\begin{proof}[Proof of Theorem \ref{CLTweak} ]   Define
\begin{align*} \theta_n^1(x)&=\psi\left(\frac{1_{P_n(x)}}{|P_n(x)|} \right)= \sum_{k=0 }^\infty   \int \phi\circ F^k  \frac{1_{P_n(x)}}{|P_n(x)|} \ dm,\\
\theta_n^2(x)&= \sum_{k<n }   \int \phi\circ F^k  \frac{1_{P_n(x)}}{|P_n(x)|} \ dm,\\
\theta_n^3(x)&= \sum_{k<n }  \int \phi \frac{1_{F^kP_n(x)}}{|F^kP_n(x)|} \ dm, \\
\theta_n^4(x)&= \sum_{k<n }  \int \phi \circ F^k(x).
\end{align*}
Since the average of $\phi$ with respect to the invariant measure $\rho m$ is zero,   by the  exponential decay of correlation in  $\mathcal{B}^s_{\infty,\infty}$ and Lemma \ref{ite3} 
\begin{align*} |\theta^1_n(x)- \theta^2_ n(x)|&=\left| \sum_{k\geq n} \int \phi\circ F^k  \frac{1_P}{|P|} \ dm \right|\\
&\leq  \sum_{k\geq n}  \left| \int \phi \  \Phi_g^k\left(\frac{1_P}{|P|} \right)\ dm   \right| \leq  \Cll{nova},\end{align*} 
Moreover for $x\in P$ 
\begin{align*} |\theta^2_n(x)-\theta^3_n(x)|&=\left| \sum_{k<n } \int \phi\circ F^k  \frac{1_P}{|P|} \ dm - \sum_{k<n }  \int \phi \frac{1_{F^kP}}{|F^kP|} \ dm  \right|\\
 &\leq \left| \sum_{k<n } \int \phi \cdot \left( \Phi_g^k  \left(\frac{1_P}{|P|}\right)  - \frac{1_{F^kP}}{|F^kP|}\right) \ dm  \right| \leq    \Cll{nova2}.\end{align*} 
 Finally if $\phi\in \mathcal{B}^s_{1,1}$ we have 
 
\begin{align*}
|\theta_n^4-\theta_n^3|_{L^1(m)}=
&\int\left| \sum_{k< n}  \left(    \phi\circ F^k(x)- \int  \phi\frac{1_{F^kP_n(x)}}{|F^kP_n(x)|} \ dm \right)  \right| \ dm(x) \\
\leq &\sum_{k< n} \int   \left|     \phi\circ F^k(x)  - \int  \phi \frac{1_{F^kP_n(x)}}{|F^kP_n(x)|} \ dm  \right| \ dm(x)\\
\leq &\sum_{k< n} \sum_{P\in\mathcal{P}^n} \int_P   \left|      \phi\circ F^k(x)  - \frac{1}{|F^kP|} \int_{F^kP}  \phi  \ dm  \right| \ dm(x)\\
\leq &\sum_{k< n} \sum_{P\in\mathcal{P}^n} \int   \left|      \phi - \frac{1}{|F^kP|} \int_{F^kP}  \phi  \ dm  \right| \Phi_g^k (1_P) \ dm\\
\leq &C \sum_{k< n} \sum_{P\in\mathcal{P}^n} \frac{|P|}{|F^kP|}   \int_{F^kP}   \left|      \phi - \frac{1}{|F^kP|} \int_{F^kP}  \phi  \ dm  \right|\ dm\\
\leq & C \sum_{k< n} \sum_{Q\in\mathcal{P}^{n-k}}   \int_{Q}   \left|      \phi - \frac{1}{|Q|} \int_{Q}  \phi  \ dm  \right|\ dm\\
\leq & C \sum_{k< n} \sum_{Q\in\mathcal{P}^{n-k}}  |Q|^{-s} \int_{Q}   \left|      \phi - \frac{1}{|Q|} \int_{Q}  \phi  \ dm  \right|\ dm\leq C |\phi|_{\scriptscriptstyle \mathcal{B}^s_{1,1}}
\end{align*}

If $\phi\in \mathcal{B}^s_{\infty,\infty}$ we  get the sharper estimate 

\begin{align*}
|\theta_n^4(x)-\theta_n^3(x)| =
&\left| \sum_{k< n}  \left(    \phi\circ F^k(x)- \int  \phi\frac{1_{F^kP_n(x)}}{|F^kP_n(x)|} \ dm \right)  \right|   \\
&\leq  C \sum_{k< n}  |F^kP_n(x)|^s \leq \Cll{n3}.
\end{align*}

Putting together those estimates, Theorem  \ref{varthm} and 
Slutsky's theorem we conclude the proof of the fist equation in the Theorem. But note that 
$$
 \psi\!\left( \frac{1_{P_n(x)}}{|P_n(x)|} \right) =  \sum_{k=0}^{n-1} c_0(\psi,P^k(x))\,\phi_{P^k(x)}(x).
$$


\end{proof}

\section{Real-analyticity of the variance and Dichotomies} 

\begin{proposition}\label{ana}  Let $t\in (a,b)\mapsto \psi_t$ be a real-analytic family of real-valued functions on either  $\mathcal{B}^s_{1,1}$, $1/2\leq s < \gamma$,   or   $\mathcal{B}^s_{\infty,\infty}$, with $s> 0$.  Then
$$t\mapsto \sigma^2(\psi_t)$$
is real-analytic. 
\end{proposition}  
\begin{proof} For the case $\psi_t \in B^s_{\infty,\infty}$ this follows from the fact that the topological pressure is real analytic  in the space of H\"older functions (Ruelle \cite{ruelle2nd}) and the variance can be expressed in terms of the second order derivatives of these function. So we will prove the result only for $\psi_t \in \mathcal{B}^s_{1,1}$,with  $1/2\leq s < \gamma$. Using the notation of the proof of Theorem \ref{varthm}., define 
 $$w_t = \sum_{k=1}^\infty T^k\psi_t. $$
We have  that 
  $$h_t=\psi_t- (w_t\circ F - w_t) $$
  is a real analytic family in $Ker \ T$. In particular
  $$t\mapsto \sigma^2(\psi_t)=\sigma^2(h_t)= |h_t|_{L^2(m)}.$$
is real analytic and we are done. 
\end{proof} 

\begin{proof}[Proof of Theorem \ref{dic2}]  Suppose  that $t\mapsto v_t\in \mathcal{B}^{\gamma}_{\infty,\infty}$ is real-analytic.  Note that 
$$t\mapsto D^\beta v_t \in \mathcal{B}^{\gamma- \beta}_{\infty,\infty}$$
is real-analytic. By Theorem \ref{main4}.C we have that $t\mapsto \alpha_t\in \mathcal{B}^\beta_{\infty,\infty}$  is real-analytic, where $\alpha_t$ is the unique bounded solution of  
$$v_t = \alpha_t\circ F - g^\beta \alpha_t.$$
By Theorems \ref{chain2} and \ref{leibnitz} the have that for small $\epsilon > 0$ the map
$$t\mapsto \psi_t:= \frac{D^\beta v_t -R_C(\alpha_t) +R_{L}(\alpha_t)  }{g^{\beta}} \in\mathcal{B}^{\gamma- \beta -\epsilon}_{\infty,\infty} $$
is also  real-analytic, so by Proposition \ref{ana}  we have that   $t\mapsto \sigma^2(\psi_t)$ is real-analytic and the result follows due Theorem \ref{dic1}.

The proof in the case of real-analytic family $t\mapsto v_t\in \mathcal{B}^{\gamma}_{1,1}$, with $\gamma-\beta > 1/2$ is similar, using Theorem \ref{chain}  and Theorem \ref{leibnitz1} instead. 
\end{proof}

\begin{proof}[Proof of Theorem \ref{dic4}]  We will prove $A$. Since  the inclusion $i\colon \mathcal{B}\rightarrow \mathcal{B}^{\gamma}_{\infty,\infty} $ is continuous  we have that 

 $$v\in \mathcal{B}  \mapsto \psi_v := \frac{D^\beta v -R_C(\alpha_v) +R_{L}(\alpha_v)  }{g^{\beta}} \in\mathcal{B}^{\gamma- \beta -\epsilon}_{\infty,\infty} $$
 
 is continuous, so 
 
 $$v\mapsto \sigma^2(\psi_v)$$
 is continuous.  It follows that for every $\epsilon >0$ the set $\mathcal{F}_\epsilon$ of functions $$\gamma \in C^k([0,1]^d,\mathcal{B})$$
 such that 
 $$m_d(\{t\in [0,1]^d\colon  \sigma^2(\psi_{\gamma(t)}) < \epsilon\}).$$
 is open in the $C^k$ topology. 
  So by the Baire Theorem it is enough to prove that $\mathcal{F}_\epsilon$ is dense. 
  
  Let  $\gamma \in C^k([0,1]^d,\mathcal{B} )$ be an arbitrary family. 
 Indeed let $w\in \mathcal{B}$ be such that $\sigma^2(\psi_{w})> 0$.  Extend the family $\gamma$ to a family $$\hat{\gamma} \in C^k([0,1]^{d+1},\mathcal{B} )$$ defined as 
 $$\hat{\gamma}(t,u)= \gamma(t)+ u w$$
 for $$(t,u)\in [0,1]^d\times [0,1].$$
 Note that  for fixed $t$ we have that $u\mapsto  \hat{\gamma}(t,u)$ is real-analytic. If $\sigma^2(\psi_{\gamma(t)})~=~0$ then
 $$\sigma^2(\psi_{\hat{\gamma}(t,u)})=u^2 \sigma^2(\psi_w)> 0$$
 for $u > 0$. If $\sigma^2(\psi_{\gamma(t)})> 0$ then  $\sigma^2(\psi_{\hat{\gamma}(t,u)})> 0$ for small $u$. In both cases by  Theorem \ref{dic2} we conclude the for every fixed $t\in [0,1]^d$ we have that 
 $$m_1(\{u\in [0,1]\colon \sigma^2(\psi_{\hat{\gamma}(t,u)})=0\})=0,$$
 so by Fubini Theorem for  almost every $u\in [0,1]$  we have 
  $$m_d(\{t\in [0,1]^d\colon \sigma^2(\psi_{\hat{\gamma}(t,u)})=0\})=0.$$
  In particular there is a sequence $u_n \rightarrow 0$ such that 
  
    $$m_d(\{t\in [0,1]^d\colon \sigma^2(\psi_{ \gamma(t)+u_n w})=0\})=0,$$
    
    so the family $t\mapsto  \gamma(t)+u_n w$ belongs to $\mathcal{F}_\epsilon$ for every $n$, and we conclude that 
    $\mathcal{F}_\epsilon$ is dense in the $C^k$ topology.

For Banach space $\mathcal{B}$ with continuous inclusions in  $\mathcal{B}^{\gamma}_{1,1}$, with  $\gamma-\beta > 1/2$ is similar, using Theorem \ref{chain}  and Theorem \ref{leibnitz1} instead.  This completes the proof of $A$.

The proof of $B$ is   easier. Let $v \in \mathcal{B}$. If  $\sigma^2(\psi_v)~=~0$ then
 $$\sigma^2(\psi _{v+uw})= \sigma^2(\psi _{v} +u \psi_w) =u^2 \sigma^2(\psi_w)> 0$$
 for $u > 0$. If $\sigma^2(\psi_{v})> 0$ then  $\sigma^2(\psi_{\hat{v}})> 0$ if $\|\hat{v}-v\|_{\mathcal{B}}$ is small smooth. 
 
 Now suppose  $\log g \in \mathcal{B}^{\gamma+\delta}_{\infty,\infty}$, with $\delta > 0$ and we are going to  prove that  $\mathcal{B}=\mathcal{B}^\gamma_{\infty,\infty}$ and  $\mathcal{B}=\mathcal{B}^\gamma_{1,1}$  satisfy the assumptions of the theorem. 
 
Choose a function  $\phi \in \mathcal{B}^{\gamma+\delta}_{\infty,\infty}$ such that 
$$\int \phi \rho \ dm =0$$ 
and such that it value on the fixed point of $F$ are no zero. Then $\phi$ is not cohomologous to a constant.   So $\sigma^2(\phi) > 0$. Let $\psi $ be the only distribution in $\mathcal{B}^{-s}_{\infty,\infty}$, for every small $s> 0$,  such that    $\phi = \psi \circ F - \psi$.  Then $D^{-\beta} \psi\in  \mathcal{B}^{Re \ \beta-s}_{\infty,\infty} $ and by Theorem \ref{chain2} and Theorem \ref{leibnitz1}  if 
 $$w:=  (D^{-\beta} \psi)\circ F - g^\beta (D^{-\beta} \psi)$$
 then 
 \begin{align*} 
 D^\beta w &= g^{\beta} \left( \psi\circ F  -\psi \right) + R_C(D^{-\beta} \psi) -R_{L}(D^{-\beta} \psi) \in\mathcal{B}^{\gamma +\delta - Re \beta  -\epsilon}_{\infty,\infty} 
 \end{align*} 
 for every small $\epsilon> 0$, so  $$w\in \mathcal{B}^{\gamma +\delta/2}_{\infty,\infty} \subset \mathcal{B}^\gamma_{\infty,\infty}\cap  \mathcal{B}^\gamma_{1,1}$$ and $\phi_w= \phi$, so $\sigma^2(\phi_w) > 0$.
\end{proof}

\begin{proposition} Let $q\in \{1,\infty\}$, $\gamma, \beta_0 \in \mathbb{C}$ and $\epsilon > 0$. Then for $\beta \in \mathbb{C}$  satisfying  $|Re \ \beta-Re \ \beta_0|< \delta$ 
$$D^\beta\colon \mathcal{B}^{\gamma}_{q,q} \rightarrow \mathcal{B}^{\gamma -  Re \ \beta_0 -\delta}_{q,q}  $$
is a bounded linear transformation and $\beta\mapsto  D^\beta$ is  complex analytic.
\end{proposition} 
\begin{proof} It is enough to show that $\beta \mapsto D^\beta$ is complex analytic on the region $|\beta -\beta_1|< \delta$ for every $\beta_1$ satisfying $Re \ \beta_1=Re \ \beta_0$.  We will prove it for $q=1$, the other case is nearly identical. For $\psi\in \mathcal{B}^{\gamma}_{1,1}$ we have 
$$\psi = c_I 1_I+ \sum_n \sum_{P\in \mathcal{P}^n} c_P |P|^{\gamma}\phi_P,$$
$$D^\beta \psi =   \sum_n \sum_{P\in \mathcal{P}^n} c_P |P|^{\gamma-\beta}\phi_P,$$
with
$$ \|\psi\|_{\mathcal{B}^{\gamma}_{1,1} }= |c_I| + \sum_n \sum_{P\in \mathcal{P}^n} |c_P|,$$

We have 
$$\left. \frac{d^k |P|^{\beta}}{d\beta^k} \right|_{\beta=\beta_1} = (\log |P|)^k |P|^{\beta_0},$$
and there is $C$ such that for every $P\in \mathcal{P}$ 

So, in the sense of distributions 

\begin{equation} \label{taylor} D^\beta \psi    = \sum_{n=0}^\infty (\beta -\beta_1)^k \frac{T_k\psi}{k!},\end{equation} 
where 
$$T_k\psi =  \sum_n \sum_{P\in \mathcal{P}^n} c_P (\log |P|)^k |P|^{\gamma-\beta_1}\phi_P. $$
Since 
\begin{align*} \|T_k\psi \|_{\mathcal{B}^{\gamma-Re \ \beta_0-\delta}_{1,1} } &=    
   \sum_n \sum_{P\in \mathcal{P}^n} |c_P| |\log |P||^k |P|^{\delta}\leq \left(\frac{k}{\delta} \right)^ke^{-k}  \|\psi\|_{\mathcal{B}^{\gamma}_{1,1} },
\end{align*} 
that implies that the radius of convergence of (\ref{taylor}) in $\mathcal{L}( \mathcal{B}^{\gamma}_{1,1} ,\mathcal{B}^{\gamma - Re \  \beta_0 -\delta}_{1,1} )$ is $\delta$. 
\end{proof}

\begin{proof}[Proof of Theorem \ref{dic3}]  First suppose $\beta \in (0,\gamma) \mapsto v_\beta \in  \mathcal{B}^\gamma_{\infty,\infty}$ is real analytic. Fix a small $\epsilon > 0$. By Theorem \ref{main4}.A there is a bounded linear operator
$$S_\beta\colon \mathcal{B}^{\beta-\epsilon}_{\infty,\infty}\rightarrow  \mathcal{B}^{\operatorname{Re} \ \beta-\epsilon}_{\infty,\infty}$$
such that $\alpha_\beta =S_\beta(v_\beta)$ is the unique bounded solution of 
$$v_\beta= \alpha_\beta\circ F - g^\beta \alpha_\beta$$
Moreover for every $\beta_0$  there is $\delta \in (0,\beta_0)$ and  $\Cll{ub}$ such that 
$$|S_{\beta}|_{\mathcal{B}^{\operatorname{Re} \beta_0-\epsilon}_{\infty,\infty}} \leq \Crr{ub}$$
for every complex  $\beta$ such that $|\operatorname{Re}  \beta-\operatorname{Re} \beta_0|< \delta$. Note that 

$$\alpha_\beta=S_\beta(v_\beta) = \sum_{k=0}^\infty  -\frac{v_\beta\circ F^k}{g^\beta_{k+1}}$$
where $$g_k^\beta(x):= \Pi_{j=0}^{k-1} g^\beta\circ F^j(x).$$
Note that  $\beta\mapsto \alpha_\beta(x) $  is $C^\infty$. 
Indeed one  can  prove by induction 

$$\partial_\beta^n v_\beta + \sum_{k=0}^{n-1} \binom{n-1}{k} (\ln g)^{n-k}g^\beta (\partial_\beta^k\alpha_\beta)  =  (\partial_\beta^n \alpha_\beta)\circ F - g^\beta (\partial_\beta^n \alpha_\beta),$$
$$\partial_\beta^n \alpha_\beta = S_\beta \left (\partial_\beta^n v_\beta + \sum_{k=0}^{n-1} \binom{n-1}{k} (\ln g)^{n-k}g^\beta (\partial_\beta^k\alpha_\beta) \right),$$
and
$$|\partial_\beta^n \alpha_\beta|_{\scriptscriptstyle \mathcal{B}^{\operatorname{Re}\beta_0-\epsilon}_{\infty,\infty}}\leq  C n! \lambda^n,$$
for some $ C > 0$ and $\lambda > 1$,  so  in particular  $\beta\mapsto \alpha_\beta \in \mathcal{B}^{\operatorname{Re}\beta_0-\epsilon}_{\infty,\infty}$ is complex analytic around $\beta_0$. So $$\beta\mapsto \phi_\beta= D^\beta \alpha_\beta \circ F -  D^\beta \alpha_{\beta}  \in \mathcal{B}^{ -\epsilon}_{\infty,\infty}$$ is also  complex analytic. In particular  for every $P\in \mathcal{P}$ we have 
$$\beta\mapsto  c_{-\epsilon} (\phi_\beta,P)$$
is complex analytic. On the other hand the  Chain and Leibnitz rules  implies that  there is $C> 0$ such that 
$$\| \phi_\beta\|_{\mathcal{B}^{\gamma - \operatorname{Re} \beta_0-\epsilon}_{\infty,\infty}}  \leq C$$
for a complex $\beta$  close to $\beta_0$, so  
 $$\beta\mapsto  c_{\gamma - \operatorname{Re} \beta_0-\epsilon} (\phi_\beta,P)= |P|^{-\gamma + \operatorname{Re} \beta_0} c_{-\epsilon} (\phi_\beta,P)$$
 is complex analytic and 
 $|c_{\gamma - \operatorname{Re} \beta_0-\epsilon} (\phi_\beta,P)|\leq C$, hence  by Cauchy's estimate it follows that 
 $\beta \mapsto \phi_\beta$ is complex analytic on $\mathcal{B}^{\gamma - \operatorname{Re} \beta_0-\epsilon}_{\infty,\infty}$. So $\beta\mapsto \sigma^2( \phi_\beta)$ is real-analytic. 

The proof for the case $v\in  \mathcal{B}^\gamma_{1,1}$ is similar, replacing Theorem \ref{main4}.A by Theorem \ref{azaz}.

In particular either $\sigma^2(\phi_\beta)=0$ for every $\beta \in (0,\gamma)$  or  just in some isolated points. By Theorems \ref{dic11} and \ref{dic1} we completed  the proof, except for the statement that $0$ is not an accumulation point  when $v_\beta \in  \mathcal{B}^\gamma_{\infty,\infty}$.

Let  $\rho_{1+\beta}$ be as in Theorem \ref{dic333}.C. Let  
$$w_{\beta}= 1- g^{\beta}= 1_I\circ F - g^{\beta} \cdot 1_I$$
and define 
\begin{equation} v_\beta =  v \int \frac{w_\beta}{g^\beta} \rho_{1+\beta} \ dm  -   w_\beta \int \frac{v }{g^\beta}  \rho_{1+\beta} \ dm. \label{lincomb} \end{equation} 
Then  $v_0=0$, $\beta \in (-\delta,\delta)\mapsto v_\beta \in \mathcal{B}^\gamma_{\infty,\infty}$  is real analytic, and 
\begin{equation*}\label{eq:integral-signs}
\begin{aligned}
\int \frac{v_\beta}{g^\beta}  \, \rho_{1+\beta} \, dm &= 0,
&\qquad &\text{for every } \beta  \sim 0 ,\\
\int \frac{w_\beta}{g^\beta}   \, \rho_{1+\beta}\, dm &< 0,
&\qquad &\text{for every } \beta > 0, \ \beta \sim 0.
\end{aligned}
\end{equation*}
Define the distribution  $\hat{\alpha}_\beta \in \mathcal{B}^{-s}_{1,1}$  as 

$$\langle \hat{\alpha}_\beta, \psi\rangle  =- \sum_{k=1}^\infty \int \frac{v_\beta\circ F^{n-1}(x)}{ g_n^\beta(x)} \psi(x) \ dm.$$
By Theorem \ref{dic333}.C we have that $ \hat{\alpha}_\beta$ is well-defined  and it is the unique solution of  $v= \hat{\alpha}_\beta \circ F - g^\beta \hat{\alpha}_\beta$ in  $\mathcal{B}^{-s}_{1,1}$ for every $\beta\sim 0$, $\beta\neq 0$.
  In particular every $L^1$  solution coincides with $\hat{\alpha}_\beta$, so for $\beta > 0$, $\beta\sim 0$ we have that $\hat{\alpha}_\beta$ coincides with the only bounded function that satisfies the twisted cohomological  equation.

As before we can prove that $\beta\mapsto \hat{\alpha}_\beta \in  \mathcal{B}^{-s}_{1,1}$ is complex analytic for every small  $s> 0$, so it follows that $\beta \mapsto \phi_{\beta, v_\beta}= (D^\beta \hat{\alpha}_\beta) \circ F - D^\beta \hat{\alpha}_\beta \in \mathcal{B}_{1,1}^{-\delta}$ is complex analytic for $\beta$ close to $0$ and  every $\delta > 0$. The Chain and Leibnitz rules (Theorems \ref{chain} and \ref{leibnitz}) given $\delta > 0$ there is $C> 0$ such that for $\beta$ close to $0$ we have 

$$\phi_{\beta, v_\beta} = \frac{D^\beta v_\beta + R_C(\hat{\alpha}_\beta) -R_{L}( \hat{\alpha}_\beta)}{g^\beta}$$
and $|\phi_{\beta, v_\beta}|_{\scriptscriptstyle \mathcal{B}^{\gamma-\delta}_{1,1}}\leq C$, that implies that $\beta\mapsto \phi_{\beta, v_\beta}$ is complex analytic in $\mathcal{B}^{\gamma-\delta}_{1,1}$, so $\beta \mapsto \sigma^2(\phi_{\beta, v_\beta})$ is real analytic   for real $\beta$ close to $0$. Finally note that due  Theorem \ref{dic11} and (\ref{lincomb}) we have that if $\beta >0$, $\beta \sim 0$ and $\beta > 0$  then $ \sigma^2(\phi_{\beta, v_\beta})=0$ if and only  if $\sigma^2(\phi_{\beta, v })=0$ . So if $0$ is an accumulation point of 
$$\Lambda=\left\{ \beta\in (0,\gamma)\colon \sigma^2(\phi_{\beta, v })=0\right\} =  \left\{ \beta\in (0,\gamma)\colon \sigma^2(\phi_{\beta, v_\beta })=0\right\}$$
it  follows that  $\Lambda =(0,\gamma)$. 
\end{proof}

\begin{proof}[Proof of \ref{mainex}] \  \\
{\it Proof of itens A-D.} Let $p\in \mathbb{S}^1$ be the fixed point of $F$. Then $F^{-n}\{p\}$ defines a partition $\mathcal{P}^n$ of $\mathbb{S}^1$ consisting on arcs. Let $m$ be the Haar measure of $\mathbb{S}^1$. $\mathcal{P}$ is a good grid for the measure space $(\mathbb{S}^1,m)$. Indeed this grid has a nicer property: it is a quasisymmetric grid, that is, there is $C> 1$ such that 

$$\frac{1}{C}\leq  \frac{|P_1|}{|P_2|}\leq C$$
for  every  adjacent $P_1,P_2\in \mathcal{P}^n$, with $n\geq 0$. Since the assumptions of  Theorem  \ref{dic4} holds for $\mathcal{B}^\gamma_{1,1}$ we have that there is $w\in \mathcal{B}^\gamma_{1,1}$ such that $\alpha_w$ {\it does not}  satisfies   I -V in Theorem  \ref{dic1}.  In particular $\sigma(D^\beta \alpha_w \circ F -  D^\beta \alpha_{w} )> 0$. Note that 
$$v\in \mathcal{B}^\gamma_{1,1} \mapsto \sigma(D^\beta \alpha_v \circ F -  D^\beta \alpha_v )$$
is continuous.  By S. \cite[Theorem 8.1]{smaniajga} we have that $\mathcal{B}^\gamma_{1,1}$, with $s > 0$, coincides with the classical Besov space $B^\gamma_{1,1}$ of $\mathbb{S}^1$ , so in particular $C^\infty(\mathbb{S}^1)$ is dense in  $\mathcal{B}^\gamma_{1,1} $, so we can replace $w$  for a  function in $C^\infty(\mathbb{S}^1)$ if necessary. As a consequence we can apply Theorems \ref{dic2}, \ref{dic4} to $\mathcal{B}=C^k(\mathbb{S}^1)$. We can also apply Theorem \ref{dic3}. 

{\it Proof of item $E.$} Note that if $\sigma(D^\beta \alpha_v \circ F -  D^\beta \alpha_{v} )> 0$  then Theorem \ref{dic11}.IV does not holds, that easily implies that $\alpha_v$ satisfies the {\it anti-H\"older condition}, that is, there is $C > 0$ such that for every arc $J\subset \mathbb{S}^1$ we have
$$\sup_{\substack{x,y\in J}}|\alpha_v(y)-\alpha_v(x)|\geq C |J|^{\beta}.$$
By Przytycki and  Urba{\'n}ski \cite[Theorem 4]{fu} we conclude E. 
\end{proof}

\begin{proof}[Proof of Theorem \ref{mainex5}] It is enough to show that if $1$ is an accumulation point of $\Lambda$ then $\Lambda=(0,1)$. 
Indeed  consider the solution  $\alpha_{1+\beta}$, with $\beta \in \mathbb{R}$, $\beta \sim 0$  for
$$v = \alpha_{1+\beta} \circ F - (DF)^{1+\beta} \cdot \alpha_{1+\beta},$$
Choose a small $\delta > 0$.   By Theorem \ref{main4}.I one can prove that 
$$\beta \mapsto \alpha_{1+\beta} \in \mathcal{B}^{1-\delta}_{\infty,\infty}$$
is real analytic.   Define
$$u_\beta= \frac{Dv + (1+\beta) (DF)^{\beta} D^2F \cdot  \alpha_{1+\beta}  }{DF} \in C^\gamma(\mathbb{S}^1)\subset \mathcal{B}^\gamma_{\infty,\infty}.$$

Note that if $\beta\sim 0$, $\beta\neq 0$ then we have that   $1$ is not  in the spectrum of $$L_{1+\beta}\colon C^{1+\gamma}(\mathbb{S}^1) \rightarrow  C^{1+\gamma}(\mathbb{S}^1) $$ so  it follows that  $\eta_\beta = D\alpha_{1+\beta}$ the only  distribution in $C^{1+\gamma}(\mathbb{S}^1)^\star$ that satisfies 

  $$u_\beta =  \eta_\beta \circ F - (DF)^\beta  \eta_\beta,$$
 
By  an argument as in the proof of Theorem \ref{dic3}, for every small $\delta> 0$ and $\epsilon> 0$  then if $\beta \neq 0$ is small enough   there is only one distributional solution $\hat{\eta}_\beta \in \mathcal{B}^{-\delta}_{1,
1}\subset C^{1+\gamma}(\mathbb{S}^1)^\star$ 
such that 
$$u_\beta = \hat{\eta}_\beta \circ F - (DF)^\beta \hat{\eta}_\beta,$$
so in particular $\hat{\eta}_\beta  =D\alpha_{1+\beta}.$ Moreover 
$$\phi_{u_\beta,\beta} = (D^\beta \hat{\eta}_\beta) \circ F - D^\beta \hat{\eta}_\beta,$$
where
$$\phi_{u_\beta,\beta}=  \frac{D^\beta u_\beta,\beta + R_C(\hat{\eta}_\beta) -R_{L}( \hat{\eta}_\beta)}{(DF)^\beta} \in \mathcal{B}^{\gamma-\epsilon}_{1,
1}$$   
  and there is a real-analytic function $\beta\mapsto \sigma^2_\beta$ defined close to $\beta =0$ such that for $\beta \sim 0$, $\beta \neq 0$ we have $\sigma^2(\phi_{u_\beta,\beta})=0$ if and only if $\sigma^2_\beta=0$.

We claim that if  $v\not\in \Omega_{1+\beta}$  for some small $\beta < 0$ then  $\alpha_{1+\beta} \in C^{1+\gamma}$ so  
$$u_\beta=  (D\alpha_{1+\beta})\circ  F -  (DF)^{\beta} \cdot  D\alpha_{1+\beta}.$$ 
so $ \hat{\eta}_\beta= D\alpha_{1+\beta},$ and consequently $D^\beta D\alpha_{1+\beta}\in \mathcal{B}^{\gamma-\beta}_{\infty,\infty}$ satisfies 
$$\phi_{u_\beta,\beta} = (D^\beta D\alpha_{1+\beta})\circ F - D^\beta D\alpha_{1+\beta}$$
which implies $\sigma^2 _\beta= \sigma^2(\phi_{u_\beta,\beta})=0$. So if $\Lambda$ has an accumulation point at $1$ we conclude that $\sigma^2(\phi_{u_\beta,\beta})=0$ for every small $\beta < 0$. That implies $D^\beta D\alpha_{1+\beta} = D^\beta \hat{\eta}_\beta \in \mathcal{B}^{\gamma-\epsilon}_{1,
1}$ and $D\alpha_{1+\beta} = \hat{\eta}_\beta \in \mathcal{B}^{\gamma+\beta -\epsilon}_{1,
1}\subset L^1,$   so $\alpha_{1+\beta}$ is almost everywhere differentiable, so it is indeed $C^1$ and $v \not\in \Omega_{1+\beta}$ for every small $\beta< 0$, and consequently for every $\beta\in (0,1)$.

To prove the last statement of the theorem we use a method similar to   Amanda  de Lima and S. \cite[Lemma 4.1 and Theorem 1.1]{ls}. Suppose that  $\beta \geq 1/2+ \gamma/2$. By Theorem \ref{main4}.I there is $C > 0$ such that 
$\|\alpha_{v,\beta}\|_{ C^{1/2+ \gamma/2}}\leq C$
for every such $\beta$. Define $\Delta^2_h u (x) = u(x+h)+u(x-h) - 2u(x)$
and
$$\omega(u,x,h) = \frac{\Delta^2_h u (x)}{|h|^{1+\gamma}}.$$
Using the Taylor series of order $2$ for $F$ we get 
$$\alpha_{v,\beta}(F(x\pm h))= \alpha_{v,\beta}(F(x)\pm DF(x)h)+ O(|h|^{1+\gamma}),$$
so 
\begin{align*} & \omega(\alpha_{v,\beta}\circ F,x,h)  = (DF(x))^{1+\gamma} \omega(\alpha_{v,\beta},F(x),DF(x)h) + O(1).\end{align*} 
On the other hand 
\begin{align*} &\alpha_{v,\beta}(F(x\pm h)) = v(x\pm h)+ (DF(x\pm h))^\beta \alpha_{v,\beta}(x\pm h) \\
&= v(x\pm h)+ \left( (DF(x))^\beta \pm   \beta (DF(x))^{\beta -1}D^2F(x)h + O(|h|^{1+\gamma})  \right) \alpha_{v,\beta}(x\pm h)  \\
&= v(x\pm h)+   (DF(x))^\beta  \alpha_{v,\beta}(x\pm h) \pm   \beta (DF(x))^{\beta -1}D^2F(x)h\alpha_{v,\beta}(x)   +  O(|h|^{1+\gamma})
\end{align*} 
and since  $v\in C^{1+\gamma}$ we obtain
$$\omega(\alpha_{v,\beta}\circ F,x,h) = (DF(x))^{\beta} \omega( \alpha_{v,\beta},x,h)+O(1).$$
We conclude that 
 $$\left| (DF(x))^{1+\gamma-\beta}  \omega(\alpha_{v,\beta},F(x),DF(x)h) - \omega( \alpha_{v,\beta},x,h)\right| \leq C_1.$$

 So let  $\lambda = \min_{x\in \mathbb{S}^1}  |Df(x)|^{1+\gamma-\beta_0}$ and 
 $$C_2= \frac{C_1}{\lambda-1}$$
 we have that if $L > 0$ then 
  $$|\omega(\alpha_{v,\beta},F(x),DF(x)h)|\geq C_2+L \implies | \omega( \alpha_{v,\beta},x,h) |\geq C_2+(DF(x))^{1+\gamma-\beta}   L.$$

  So suppose that  there is $\beta_n \in (0,1)$, with  $v\not\in \Omega_{\beta_n}$ and $\lim_n \beta_n =1$.  Then $$ |\omega( \alpha_{v,\beta_n},x,h) |\leq C_2$$
for every $x\in \mathbb{S}^1$ and $h> 0$. Indeed if $|\omega(\alpha_{v,\beta},y_0,h_0)| > C_2$ for some $y_0$ and $h_0 > 0$ then  for every $y_n$ such that $F^n(y_n)=y_0$ we have 
  $$\omega\left(  \alpha_{v,\beta}, y_n, \frac{h_0}{DF^n(y_n)} \right) \geq L (DF^n(y_n))^{1+\gamma-\beta} $$
so there is $C> 0$ such that 
$$\Delta\left(  \alpha_{v,\beta}, y_n, \frac{h_0}{DF^n(y_n)} \right) \geq L (DF^n(y_n))^{1+\gamma-\beta}  \left( \frac{h_0}{DF^n(y_n)} \right)^{1+\gamma} \geq C \left( \frac{h_0}{DF^n(y_n)} \right)^{\beta}.$$
So by Theorem \ref{dic11} it follows that $v \in \Omega_\beta$, that contradicts the assumption on $v$.

This implies $ |\omega( \alpha_{v,1},x,h) |\leq C_2$ for every $x\in \mathbb{S}^1$ and $h> 0$ and consequently $\alpha_{v,1}\in C^{1+\gamma}$. It follows from \cite{ls} that $(Dv+D^2F
\cdot \alpha_{v,1})/DF$ is cohomologous to zero, that contradicts the  assumption.
\end{proof} 

\begin{proof}[Proof of Theorem \ref{MTB}]
We recall that the classical Besov space $B^s_{1,1}(\mathbb{S}^1)$ coincides with the Besov space considered here \cite{smaniajga}, and that $C^\infty$ functions are dense in this space. Therefore, the argument used in the proof of Theorem~\ref{mainex} applies here as well, \emph{mutatis mutandis}.
\end{proof}

\section{Appendix}

\subsection{Notation}

Elements of the grid will be denoted by capital letters such as $P$, $Q$, $W$, $J$, and so on. We will use the notation
$$
\phi_P^s = |P|^s \phi_P, \quad 1_P^s = |P|^s 1_P,
$$
where $|P|$ denotes the measure of the element $P$ of the grid, and $\phi_P$ and $1_P$ denote functions supported on $P$. Real and complex numbers will be denoted by lowercase letters such as $z$, $c$, $d$, etc. These scalars will often be indexed by elements of the grid, as in $z_P$, $c_Q$, $d_W$, and so forth.

In many proofs throughout the appendix, we will frequently manipulate expressions involving multiple nested sums indexed by elements of the grid $\mathcal{P}$. To simplify the notation and improve readability, we adopt a variation of the {\it Einstein summation convention}, adapted to our grid context. According to this convention, whenever a capital letter such as $P$, $Q$, $J$, $W$, etc., appears as an index, it is understood to be a {\it bound variable} of an implicit summation over elements of the grid. For example, instead of writing the full expression
$$
\sum_{k=0}^\infty \sum_{P\in \mathcal{P}^k} c_P |P|^s \phi_P,
$$
we will simply write
$$
c_P \phi_P^s.
$$

When multiple nested sums appear, we group terms into {\it blocks}, and the position of each implicit summation is determined by the first appearance of the corresponding index. For example, we write
$$
c_P \phi_P^s z_Q 1_Q^\gamma
$$
instead of
$$
\sum_{k=0}^\infty \sum_{P \in \mathcal{P}^k} c_P |P|^s \phi_P \left( \sum_{j=0}^\infty \sum_{Q \in \mathcal{P}^j} z_Q |Q|^\gamma 1_Q \right).
$$
When we wish to restrict an inner sum to elements $Q$ that are {\it strictly}  contained in the immediately preceding bound variable $P$, we indicate this using a \textit{subscript modifier}, such as $Q_+$. For instance,
$$
c_P \phi_P d_{Q_+} 1_{Q_+}
$$
stands for
$$
\sum_{k=0}^\infty \sum_{P \in \mathcal{P}^k} c_P \phi_P \left( \sum_{j > k_0(P)} \sum_{\substack{Q \in \mathcal{P}^j \\ Q \subset P}} d_Q 1_Q \right),
$$
where $k_0(P)$ denotes the unique level such that $P \in \mathcal{P}^{k_0(P)}$. Analogously, we write
$$
c_P \phi_P d_{Q_-} 1_{Q_-}
$$
to represent
$$
\sum_{k=0}^\infty \sum_{P \in \mathcal{P}^k} c_P \phi_P \left( \sum_{j < k_0(P)} \sum_{\substack{Q \in \mathcal{P}^j \\ P \subset Q}} d_Q 1_Q \right).
$$
In this case, the second sum is {\it finite}. We also introduce further subscripts to indicate others {\it restricted finite summations}. For instance, $P_{\text{ch}}$ will denote a summation over the {\it children} of $P$ (just two elements), and we write
$$
c_P \phi_P d_{P_{\text{ch}}}
$$
to mean
$$
\sum_{k=0}^\infty \sum_{P \in \mathcal{P}^k} c_P \phi_P \left( \sum_{\substack{Q \in \mathcal{P}^{k_0(P)+1} \\ Q \subset P}} d_Q \right).
$$

Similarly, we denote by $P_\star$ the {\it pre-images}  of $P$ under the map $F$, that is, the elements $Q_i \in \mathcal{P}^{k_0(P)+1}$, $i=1,2$,  such that $F(Q_i) = P$. Then,
$$
c_P \phi_P d_{P_\star} 1_{P_\star}
$$
stands for
$$
\sum_{k=0}^\infty \sum_{P \in \mathcal{P}^k} c_P \phi_P \left( \sum_{i=1}^2 d_{Q_i} 1_{Q_i} \right).
$$

Notice that the subscripts $\mathrm{ch}$ and $\star$ follow a slightly different syntax compared to the modifiers $+$ and $-$. In these cases, we refer explicitly to the immediately preceding bound variable, rather than introducing a new capital letter. For example, $P_{\mathrm{ch}}$ is always preceded by the bound variable $P$, $Q_\star$ by $Q$, and so on. We denote by  $par(Q)$   the {\it parent} of $Q$, and $par(Q)$ must {\it not} be seem as a new bound variable.

Finally, we use the notation $O_P$ to denote a quantity that may depend on the element $P$ and on other bound variables that appear {\it  to its left } in the expression. However, this quantity is always {\it uniformly bounded}  by a constant $C$ that is independent of $P$ and of all summation indices. For example,
$$
c_P \phi_P |P|^{-1} d_{P_{\text{ch}}} |P_{\text{ch}}|
= c_P \phi_P d_{P_{\text{ch}}} O_{P_{\text{ch}}},
$$
since the ratio $|P_{\text{ch}}| / |P|$ is uniformly bounded, depending only on the geometric structure of the grid.

In sequences of manipulations using this notation, the symbols $O_P$ may refer to different quantities on opposite sides of an equality. Thus, an expression such as $exp_1 = exp_2$ should be interpreted as shorthand for “$exp_1$ can be viewed as $exp_2$”—not as a genuine equality. In particular, this is not a symmetric relation.

 In our notation, an expression involving only the bound variable \( P \) and constants will be referred to as a \textit{block}, and will be denoted by \( \omega_P \). 

For example,  the expression
\[
c_P \phi_P d_{Q_-} 1_{Q_-},
\]
can be decomposed in  two blocks: \( \omega_P = c_P \phi_P \) and \( \omega_{Q_-} = d_{Q_-} 1_{Q_-} \). Blocks admit natural transformations, such as adding or removing capital letters and /or subscript modifiers uniformly across all their occurrences.  We will use $\omega_P,  \tilde{\omega}_P$ to denote blocks that depend  on $P$. 

We will frequently use the following properties in the manipulation of multiple nested summations. Most of these are straightforward and follow directly from the structure of the notation.

\begin{proposition} Whenever such expressions appear as subexpressions within manipulations involving multiple nested summations, the following substitutions may be applied:
\begin{enumerate}[label=\alph*.]
\item  Terms  inside a block $\omega_P$ are commutable. 
\item 
$|P|^\alpha \phi_P^\beta = |P|^{\alpha+\beta} \phi_P=  \phi_P^{\alpha+\beta}$ and $|P|^\alpha 1_P^\beta = |P|^{\alpha+\beta} 1_P=  1_P^{\alpha+\beta}$.
\item For every $\alpha,\beta \in \mathbb{R}$ we have $$|P|^{\alpha}|P_{ch}|^\beta = |P|^{\alpha +\beta} O_{P_{ch}}=|P_{ch}|^{\alpha +\beta} O_{P_{ch}}$$   $$|P|^{\alpha}|P_{\star}|^\beta = |P|^{\alpha +\beta} O_{P_{\star}}=|P_{\star}|^{\alpha +\beta} O_{P_{\star}}$$
\item  $\omega_P \omega_{W_+}= \omega_W \omega_{P_-}$ and $\omega_P \omega_{W_-}= \omega_W \omega_{P_+}$.
\item  We have
$\omega_P \tilde{\omega}_Q = \omega_P\tilde{\omega}_{Q_+}  + \omega_P \tilde{\omega}_P  + \omega_P\tilde{\omega}_{Q_-} .$
\item  If $\beta > 0$ we have 
$$\omega_P|Q_+|^{1+\beta} = \omega_P|P|^{1+\beta}O_P.$$
\item  If $\beta > 0$ we have 
$$\omega_P|Q_-|^{\beta} = \omega_P O_P,$$
$$\omega_P|Q_-|^{-\beta} = \omega_P |P|^{-\beta} O_P,$$
\item  Let  $\sum_{P\in \mathcal{P}} |z_P|\leq L$.  If $\alpha \in \mathbb{R}$  then  
$$\big| O_Pz_P\phi_P^{\alpha}    \big|_{\mathcal{B}^\alpha_{1,1}}\leq CL,$$
and if $\alpha < 0$ 
$$\big| O_Pz_P1_P^{-1}    \big|_{\mathcal{B}^\alpha_{1,1}}\leq CL.$$
for some $C$ that depends only on $\sup_P |O_P|.$
\item  If   \ $\sum_{P\in \mathcal{P}} |z_P|\leq L$ then  
$$\big| O_W \phi_W |W|^s O_{P_+} |P_+|^{\gamma-s}  z_{par(P_+)}\big|_{\mathcal{B}^\alpha_{1,1}}\leq CL$$ 
for some $C$ that depends only on $\sup_P |O_P|$ and $\sup_W |O_W|$. 
\item  If  \ $\sup_{P\in \mathcal{P}} |z_P|\leq L$ then  
$$\big| O_Pz_P\phi_P^{\alpha+1}    \big|_{\mathcal{B}^\alpha_{\infty,\infty}}\leq CL,$$
$$\big| O_Pz_P1_P^{\alpha}    \big|_{\mathcal{B}^\alpha_{\infty,\infty}}\leq CL.$$
for some $C$ that depends only on $\sup_P |O_P|.$
\item  We have $|z_{\mathop{\mathrm{par}}(P)}|\leq 2 |z_{P}|.$
\item We have
$\omega_P \tilde{\omega}_{P_{ch}} \hat{\omega} = \tilde{\omega}_P \omega_{\mathop{\mathrm{par}}(P)}     \hat{\omega},$
where $\omega_{\mathop{\mathrm{par}}(P)}$ must be seen as part of the block $\tilde{\omega}_P \omega_{\mathop{\mathrm{par}}(P)}$  of the bounded variable $P$.
\item We have
$\omega_P \tilde{\omega}_{P_{\star}} \hat{\omega} = \tilde{\omega}_P \omega_{F(P)}     \hat{\omega},$
where $\omega_{\mathop{\mathrm{F}}(P)}$ must be seen as part of the block $\tilde{\omega}_P \omega_{\mathop{\mathrm{F}}(P)}$  of the bounded variable $P$.
\end{enumerate}
\end{proposition}

\subsection{Relations between Besov spaces} The proof of the following is easy.
\begin{proposition} For $  s_1< s_2$ we have that $\mathcal{B}^{s_2}_{1,1}\subset \mathcal{B}^{s_1}_{1,1}$ and $\mathcal{B}^{s_2}_{\infty,\infty}\subset \mathcal{B}^{s_1}_{\infty,\infty}$ and those inclusions are continuous. 
\end{proposition}

\begin{proposition} \label{oi} For $\beta <  \gamma$ we have $\mathcal{B}^{\gamma}_{\infty,\infty}\subset  \mathcal{B}^{\beta}_{1,1}$.  Moreover there is $C$ such that  if $P\in \mathcal{P}^{k_0}$ and 
 $$\psi =  \sum_{k\geq k_0}  \sum_{\substack{J\in \mathcal{P}^k\\ J\subset P} } c_J |J|^{\gamma+1} \phi_J, $$
so 
 $$|\psi|_{\mathcal{B}^{\gamma}_{\infty,\infty}}=\sup_{k\geq k_0} \sup_{\substack{J\in \mathcal{P}^k\\ J\subset P} }   |c_J|,$$
then 
$$|\psi|_{\mathcal{B}^{\beta}_{1,1}}\leq C |P|^{1+\gamma -\beta}|\psi|_{\mathcal{B}^{\gamma}_{\infty,\infty}}.$$

\end{proposition} 
\begin{proof} Note that 
\begin{align*} \psi =  |P|^{\gamma-\beta} \sum_{k\geq k_0}  \sum_{\substack{J\in \mathcal{P}^k\\ J\subset P} } \Big( c_J \frac{|J|^{\gamma-\beta} }{|P|^{\gamma-\beta}} |J| \Big)  |J|^{\beta} \phi_J,
\end{align*} 
and
\begin{align*}  |P|^{\gamma-\beta} \sum_{k\geq k_0}  \sum_{\substack{J\in \mathcal{P}^k\\ J\subset P} } |c_J|  \frac{|J|^{\gamma-\beta} }{|P|^{\gamma-\beta}} |J|  & \leq |P|^{\gamma-\beta}  |\psi|_{\mathcal{B}^{\gamma}_{\infty,\infty}} \sum_{k\geq k_0} |P|   \lambda_2^{(k-k_0)(\gamma-\beta)} \\
& \leq C  |P|^{1+\gamma-\beta}  |\psi|_{\mathcal{B}^{\gamma}_{\infty,\infty}}.
\end{align*} 
\end{proof}

\subsection{Dirac's approximations and  fractional derivatives}
Given $W\in \mathcal{P}^k$ let $W^o$ be  the  sister of $W$, that is, the distinct element of $\mathcal{P}^k$ that has the same parent than $W$.  One can easily prove that  for fixed $P$ (see S. \cite{smania}) 
 
\begin{equation}\label{pipi} \frac{1_P}{|P|}  =\frac{1_I}{|I|} +   |P|^0 O_{W_{-}}  \phi_{W_-}\end{equation}

\begin{proposition} \label{poi} For  $s> 0$  there is $C > 0$ such that if  
$$\phi =  c_P  1_P^{s} $$
with $L= \sup_\ell \sup_{P\in \mathcal{P}^\ell} |c_P|< \infty$ then
$$|\phi|_{\mathcal{B}^{s}_{\infty,\infty}}\leq CL.$$
\end{proposition} 
\begin{proof} Indeed
\begin{align*} c_P  1_P^s =    c_P  |P|^{1+s}  O_{W_{-}}  \phi_{W_-}  
=    O_{W}  \phi_{W}    c_{P_+} |P_+|^{1+s} =  O_{W}  L \phi_{W}^{1+s}.   \end{align*} 
so  $\big| c_P  1_P^s \big|_{\mathcal{B}^{s}_{\infty,\infty}}     \leq C L.$
\end{proof} 

\begin{proposition} \label{deriapp} We have  for fixed $P$ 
\begin{equation}\label{diracder} 
D^\beta 1_P^{-1}  =   |P|^0  O_{W_{-}}  \phi_{W_-}^{-\beta}.
\end{equation}
Moreover  given  $s,\beta \in \mathbb{R}$,  there is $C$ such  that  
$$\Big|  D^\beta  1_P^{-1} \Big|_{\mathcal{B}^s_{1,1}} \leq   C (\ln |P|)^{(\beta+s)_0}  |P|^{-(\beta+s)_+}. $$
\end{proposition} 
\begin{proof} Note that (\ref{diracder}) follows from direct derivation. In particular 
$$D^\beta  1_P^{-1} =  |P|^0  O_{W_{-}}  \phi_{W_-}^{-\beta} =  |P|^0  O_{W_{-}}  |W_-|^{-(\beta+s)} \phi_{W_-}^{s} ,$$
so
$$\Big|  D^\beta  1_P^{-1} \Big|_{\mathcal{B}^s_{1,1}} \leq C  |P|^0    |W_-|^{-(\beta+s)}  \leq C (\ln |P|)^{(\beta+s)_0}  |P|^{-(\beta+s)_+}.$$
\end{proof}

\subsection{ Operators acting  on distribution spaces }  There are two operators that are important in our studies
\begin{proposition}[Multipliers] Given $0< s< \gamma$ and $g\in \mathcal{B}^\gamma_{\infty,\infty}$ we have that 
$$\phi \in \mathcal{B}^{-s}_{1,1}\mapsto  M_g(\psi)= \phi g \in \mathcal{B}^{-s}_{1,1}$$
is a bounded operator in $\mathcal{B}^{-s}_{1,1}$. 
\end{proposition} 
\begin{proof}  Let $\phi = z_{p} \phi_P^{-s}$. By Theorem \ref{fbeta} we have
 $$z_Pg1_P = z_Pm(g,P)1_P  +   z_Pc_{J_+}(g)   \phi_{J_+}^{\gamma+1} +z_Pc_P(g)   \phi_P^{\gamma+1}$$
 with $|m(g,P)|\leq C|g|_{\mathcal{B}^\gamma_{\infty,\infty}}$ and $|c_J(g)|\leq |g|_{\mathcal{B}^\gamma_{\infty,\infty}}.$ So
  \begin{align*}  &z_P\phi_P^{-s}g  = z_P m(g,P)  \phi_P^{-s}  +   z_P|P|^{-s}c_{J_+}(g)  \phi_{J_+}^{1+\gamma}  + z_P |P|^{-s+1+\gamma} c_P(g)   1_{P_{ch}}^{-2}  \\
  &= O_Pz_P   \phi_P^{-s}|g|_{\mathcal{B}^\gamma_{\infty,\infty}}     + z_P|P|^{-s}  O_{J_+} |J_+|^{1+\gamma+s} \phi_{J_+}^{-s} |g|_{\mathcal{B}^\gamma_{\infty,\infty}}    +z_P  O_{P_{ch}}   1_{P_{ch}}^{\gamma-s-1}|g|_{\mathcal{B}^\gamma_{\infty,\infty}}    \\
&= O_Pz_P   \phi_P^{-s}|g|_{\mathcal{B}^\gamma_{\infty,\infty}}     +    O_{J}  \phi_{J}^{-s}|J|^{1+\gamma+s} O_{P_-} |P_-|^{-s} |\phi|_{\mathcal{B}^s_{1,1}}   |g|_{\mathcal{B}^\gamma_{\infty,\infty}}     +O_{P}   1_{P}^{\gamma-s-1}z_{\mathop{\mathrm{par}}(P)} |g|_{\mathcal{B}^\gamma_{\infty,\infty}} \\
&= O_Pz_P   \phi_P^{-s}|g|_{\mathcal{B}^\gamma_{\infty,\infty}}     +    O_{J}  \phi_{J}^{-s} |J|^{1+\gamma} |\phi|_{\mathcal{B}^s_{1,1}}   |g|_{\mathcal{B}^\gamma_{\infty,\infty}}     + O_{P}   1_{P}^{\gamma-s-1}z_{\mathop{\mathrm{par}}(P)} |g|_{\mathcal{B}^\gamma_{\infty,\infty}}.
  \end{align*}

 Note that 

 $$\Big|O_Pz_P   \phi_P^{-s}|g|_{\mathcal{B}^\gamma_{\infty,\infty}}  \Big|_{\mathcal{B}^{-s}_{1,1}}\leq C|g|_{\mathcal{B}^\gamma_{\infty,\infty}} |\phi|_{\mathcal{B}^s_{1,1}} .$$
  $$\Big|O_{J}  \phi_{J}^{-s} |J|^{1+\gamma} |\phi|_{\mathcal{B}^s_{1,1}}   |g|_{\mathcal{B}^\gamma_{\infty,\infty}}    \Big|_{\mathcal{B}^{-s}_{1,1}}\leq C   |g|_{\mathcal{B}^\gamma_{\infty,\infty}}  |\phi|_{\mathcal{B}^s_{1,1}}.$$
  $$\Big| O_{P}   1_{P}^{\gamma-s-1}z_{\mathop{\mathrm{par}}(P)} |g|_{\mathcal{B}^\gamma_{\infty,\infty}}  \Big|_{\mathcal{B}^{-s}_{1,1}} \leq C   |g|_{\mathcal{B}^\gamma_{\infty,\infty}}  |\phi|_{\mathcal{B}^s_{1,1}}.$$
\end{proof}

\begin{lemma} \label{koo} For fixed $P$ we have 
\begin{align*} &\phi_P\circ F =\  C_{P_\star} \phi_{P_\star}+E_{P_\star}\frac{1_{P_\star}}{|P_\star|}\end{align*} 
where  
  \begin{equation}\label{wwww} \sup_P  |C_{P_\star}|< \infty \text{ and } |E_{P_\star}|\leq C|P|^{\gamma}.\end{equation} 

\end{lemma} 
\begin{proof} Indeed, note that if $Q_1$ and $Q_2$ are the children of $P$ 
$$\phi_P\circ F=   \frac{1_{Q_1}\circ F}{|Q_1|} -\frac{1_{Q_2}\circ F}{|Q_2|}. $$
 So
\begin{align*} &\phi_P\circ F=\sum_{j=0}^1 \Big( \frac{1_{Q_1^j} }{|Q_1|} -\frac{1_{Q_2^j}}{|Q_2|} \Big) =\sum_{j=0}^1 C_P^j\Big( \frac{1_{Q_1^j}}{|Q_1^j|}  -   \frac{1_{Q_2^j}}{|Q_2^j|}\Big)   +E_P^j \frac{1_{P^j}}{|P^j|}=\sum_{j=0}^1 C_P^j \phi_{P^j}+E_P^j \frac{1_{P^j}}{|P^j|}.
 \end{align*} 
Here 
 $$|C_P^j|= \frac{|P| |Q_2^j| |Q_1^j|}{|Q_1||Q_2||P^j|} \text { and } E_P^j=\frac{|Q_1^j|}{|Q_1|} -\frac{ |Q_2^j|}{|Q_2|}.$$
 Note that due the bounded distortion of the map $F$ and the H\"older continuity of $g$  we have (\ref{wwww}).
\end{proof}

\begin{proposition}[Koopmann operator] Let $F$ be  an  expanding map as in the introduction. Given $0< s <  \gamma$ we have that 
$$\phi \in \mathcal{B}^{-s}_{1,1}\mapsto  \mathcal{K}(\phi)= \phi\circ F \in \mathcal{B}^{-s}_{1,1}$$
is a bounded operator in $\mathcal{B}^{-s}_{1,1}$. 
\end{proposition} 
\begin{proof}
 It is enough to prove that there exists $C$ such that 
$$|\mathcal{K}(|P|^{-s}\phi_P)|_{\mathcal{B}^{-s}_{1,1}}\leq C$$
for every $P\in \cup_k  \mathcal{P}^{k_0}$ and every $k_0$. Since there is $C$ such that 
$$ \frac{|P^j|^s}{|P|^s} \leq C,$$
by Lemma \ref{koo}  and Proposition \ref{deriapp} 
\begin{align*} \Big| |P|^{-s} \phi_P\circ F \Big|_{\mathcal{B}^{-s}_{1,1}} &=\Big| \sum_{j=0}^1 \Big( C_P^j  \frac{|P^j|^s}{|P|^s}    |P^j|^{-s} \phi_{P^j}+  E_P^j \frac{|P^j|^s}{|P|^s}  |P^j|^{-s} \frac{1_{P^j}}{|P^j|} \Big)\Big|_{\mathcal{B}^{-s}_{1,1}}\leq C.\end{align*} 
\end{proof}

By Theorem \ref{fbeta} we have
 \begin{equation}  g^{\beta}1_{P^j} = \big(\frac{|P|}{|P^j|}\big)^\beta1_{P^j} + d_{P^j} 1_{P^j}  +  \sum_{ J\subset P^j}  c_J |J|^{\gamma+1} \phi_J, \end{equation} 

 with $|d_{P^j}|\leq C|g^{\beta}|_{\mathcal{B}^\gamma_{\infty,\infty}}|P^j|^\gamma$ and $|c_J|\leq |g^{\beta}|_{\mathcal{B}^\gamma_{\infty,\infty}}.$
Consequently by Lemma \ref{koo}

\begin{lemma} \label{q23q23} For every $s,\beta$  \begin{align*} 
  & z_P(D^\beta |P|^s\phi_P)\circ Fg^{\beta}  =  z_PD^\beta (|P|^s\phi_P\circ F) \\
  &+ \overbrace{  z_P O_{P_{\star}}  |P_{\star}|^{s+\gamma}D^\beta(\frac{1_{P_{\star}}}{|P_{\star}|})}^{I_{s,\beta}} +\overbrace{   z_PO_{P_{\star}}    1_{P_{\star}}^{s-\beta+\gamma-1}    }^{   II_{s,\beta}} \nonumber   \\
&+\overbrace{   z_P O_{P_{\star}}  \phi_{P_{\star}}^{s-\beta+\gamma}   +   z_PO_{P_{\star}}  \phi_{P_{\star}}^{s-\beta}   c_{J_+}  \phi_{J_+}^{\gamma+1}}^{III_{s,\beta}}   +\overbrace{   z_PO_{P_{\star}}  |P_{\star}|^{s-\beta+\gamma-1}   c_{J_+}  \phi_{J_+}^{\gamma +1} }^{IV_{s,\beta}}. \nonumber 
 \end{align*} 
\end{lemma}

\subsection{Proof of Chain and Leibniz product rules}
Since $||P|^\beta| = |P|^{\operatorname{Re} \beta}$, all estimates that hold for $\operatorname{Re} \beta$ also hold for $\beta$. Hence, without loss of generality, we may assume that $\beta \in \mathbb{R}$.

\begin{proof}[Proof of Chain Rule I (Theorem \ref{chain2})]
Every function $\psi \in \mathcal{B}_{\infty,\infty}^s$ can be written as
$$
\psi = z_0 1_I + z_P \phi_P^{s+1},
$$
where
$$
|\psi|_{\scriptscriptstyle \mathcal{B}^s_{\infty, \infty}} = |z_0| + \sup_k \sup_{P \in \mathcal{P}^k} |z_P|.
$$
Note that $D^\beta (z_01_I\circ F)= D^\beta (z_01_I)\circ Fg^\beta =0$. 
By Lemma \ref{koo}, we have  \change{we used $\gamma + s> 0$} 
\begin{align*}
z_P D^\beta(\phi_P^{s+1} \circ F)
&= z_P |P|^{s+1} C_{P_\star} |P_\star|^{-\beta} \phi_{P_\star}
+ z_P |P|^{\gamma+s+1} O_{P_\star} D^\beta \, 1_{P_\star}^{-1} \notag \\
&= z_P |P|^{s+1} C_{P_\star} |P_\star|^{-\beta} \phi_{P_\star}
+ O_P |P|^{\gamma+s+1} O_{P_\star} O_{W_{-}} \phi_{W_-}^{-\beta} |\psi|_{\scriptscriptstyle \mathcal{B}^s_{\infty, \infty}} \notag \\
&= z_P |P|^{s+1} C_{P_\star} |P_\star|^{-\beta} \phi_{P_\star}
+ O_W \phi_W^{-\beta} O_{P_+} |P_+|^{\gamma+s+1} |\psi|_{\scriptscriptstyle \mathcal{B}^s_{\infty, \infty}} \notag \\
&= z_P |P|^{s+1} C_{P_\star} |P_\star|^{-\beta} \phi_{P_\star}
+ O_W \phi_W^{\gamma+s-\beta+1} |\psi|_{\scriptscriptstyle \mathcal{B}^s_{\infty, \infty}}.
\end{align*}
Therefore,
$$
\left| O_W \phi_W^{\gamma+s-\beta+1} |\psi|_{\scriptscriptstyle \mathcal{B}^s_{\infty, \infty}} \right|_{\scriptscriptstyle \mathcal{B}^{\gamma+s-\beta}_{\infty, \infty}} \leq C |\psi|_{\scriptscriptstyle \mathcal{B}^s_{\infty, \infty}}.
$$
Moreover,
\begin{align*}
\left| z_P |P|^{s+1} C_{P_\star} |P_\star|^{-\beta} \phi_{P_\star} \right|_{\scriptscriptstyle \mathcal{B}^{s-\beta}_{\infty, \infty}} 
= \left| O_P \phi_P^{s+1-\beta} |\psi|_{\scriptscriptstyle \mathcal{B}^s_{\infty, \infty}} \right|_{\scriptscriptstyle \mathcal{B}^{s-\beta}_{\infty, \infty}} 
\leq C |\psi|_{\scriptscriptstyle \mathcal{B}^s_{\infty, \infty}}.
\end{align*}
By Lemma \ref{q23q23}, we also have
\begin{equation*}
z_P \left(D^\beta \phi_P^{s+1}\right) \circ F \cdot g^\beta
= z_P D^\beta(\phi_P^{s+1} \circ F)
- I_{s+1,\beta} + II_{s+1,\beta} + III_{s+1,\beta} + IV_{s+1,\beta}.
\end{equation*}
Now we estimate each term. Using  (\ref{diracder}), \change{we used $s-\beta+\gamma> 0$} 
\begin{align*}
I_{s+1,\beta}
&= z_P O_{P_{\star}} |P_{\star}|^{s+\gamma+1} D^\beta\left( 1_{P_{\star}}^{-1} \right) \nonumber\\
&= |\psi|_{\scriptscriptstyle \mathcal{B}^s_{\infty,\infty}} O_P |P|^{s+\gamma+1} D^\beta\left(1_P^{-1}\right) \nonumber\\
&= |\psi|_{\scriptscriptstyle \mathcal{B}^s_{\infty,\infty}} O_W \phi_W^{s-\beta+\gamma+1}.
\end{align*}
Thus,
$$
\left| I_{s+1,\beta} \right|_{\scriptscriptstyle \mathcal{B}^{s-\beta+\gamma}_{\infty,\infty}} \leq C |\psi|_{\scriptscriptstyle \mathcal{B}^s_{\infty,\infty}}.
$$
For the second term,\change{we used $s-\beta+\gamma> 0$} 
\begin{align*}
\left| II_{s+1,\beta} \right|_{\scriptscriptstyle \mathcal{B}^{s-\beta+\gamma}_{\infty,\infty}} 
= \left| z_P O_{P_{\star}} 1_{P_{\star}}^{s-\beta+\gamma} \right|_{\scriptscriptstyle \mathcal{B}^{s-\beta+\gamma}_{\infty,\infty}} 
\leq C |\psi|_{\scriptscriptstyle \mathcal{B}^s_{\infty,\infty}}.
\end{align*}
Now for the third term,
\begin{align*}
III_{s+1,\beta}
&= z_P O_{P_{\star}} \phi_{P_{\star}}^{s-\beta+\gamma+1}
+ z_P O_{P_{\star}} \phi_{P_{\star}}^{s-\beta+1} c_{J_+} \phi_{J_+}^{\gamma+1} \notag \\
&= O_P \phi_P^{s-\beta+\gamma+1} z_{F(P)}
+ O_P \phi_P^{s-\beta+1} z_{F(P)} c_{J_+} \phi_{J_+}^{\gamma+1} \notag \\
&= O_P \phi_P^{s-\beta+\gamma+1} |\psi|_{\scriptscriptstyle \mathcal{B}^s_{\infty,\infty}} 
+ O_J \phi_J^{(s-\beta)_- - \epsilon + \gamma + 1} |\psi|_{\scriptscriptstyle \mathcal{B}^s_{\infty,\infty}} |g^\beta|_{\scriptscriptstyle \mathcal{B}^\gamma_{\infty,\infty}}.
\end{align*}
Therefore,
$$
\left| III_{s+1,\beta} \right|_{\scriptscriptstyle \mathcal{B}^{\gamma + (s-\beta)_- - \epsilon}_{\infty,\infty}} \leq C |\psi|_{\scriptscriptstyle \mathcal{B}^s_{\infty,\infty}} |g^\beta|_{\scriptscriptstyle \mathcal{B}^\gamma_{\infty,\infty}}.
$$
Finally, for the last term,
\begin{align*}
IV_{s+1,\beta}
&= z_P O_{P_{\star}} |P_{\star}|^{s-\beta+\gamma} c_{J_+} \phi_{J_+}^{\gamma+1} \notag \\
&= O_J \phi_J^{\gamma+1} |\psi|_{\scriptscriptstyle \mathcal{B}^s_{\infty,\infty}} |g^\beta|_{\scriptscriptstyle \mathcal{B}^\gamma_{\infty,\infty}},
\end{align*}
so that
$$
\left| IV_{s+1,\beta} \right|_{\scriptscriptstyle \mathcal{B}^{\gamma}_{\infty,\infty}} \leq C |\psi|_{\scriptscriptstyle \mathcal{B}^s_{\infty,\infty}} |g^\beta|_{\scriptscriptstyle \mathcal{B}^\gamma_{\infty,\infty}}.
$$
\end{proof}

\begin{proof}[Proof of Chain Rule II (Theorem \ref{chain})]
Every function $\psi \in \mathcal{B}_{1,1}^s$ can be written as
$$
\psi = z_0 1_I + z_P \phi_P^s,
$$
where
$$
|\psi|_{\scriptscriptstyle \mathcal{B}^s_{1,1}} = |z_0| + |z_P|.
$$
Note that $D^\beta (z_01_I\circ F)= D^\beta (z_01_I)\circ Fg^\beta =0$. 
By Lemma \ref{koo}, we have
\begin{align*}
z_P D^\beta(\phi_P^s \circ F)
&= z_P |P|^s C_{P_\star} |P_\star|^{-\beta} \phi_{P_\star}
+ z_P |P|^{\gamma + s} O_{P_\star} D^\beta 1_{P_\star}^{-1} \notag \\
&= z_P |P|^s C_{P_\star} |P_\star|^{-\beta} \phi_{P_\star}
+ |P|^{\gamma + s} O_P z_{F(P)} D^\beta 1_P^{-1}.
\end{align*}
By Proposition \ref{deriapp}, we estimate the second term: \change{$s+\gamma > 0$} 
$$
\left| |P|^{\gamma + s} O_P z_{F(P)} D^\beta 1_P^{-1} \right|_{\scriptscriptstyle \mathcal{B}^{s - \beta + \gamma}_{1,1}} 
\leq C |z_{F(P)}|  C (\ln |P|)^{(s+\gamma)_0}  |P|^{(s+\gamma)_-}  \leq C |\phi|_{\scriptscriptstyle \mathcal{B}^s_{1,1}},
$$
since $s +\gamma > 0$. On the other hand, by Lemma \ref{q23q23}, we have
$$
z_P \left(D^\beta (|P|^s \phi_P)\right) \circ F \cdot g^\beta 
= z_P D^\beta(\phi_P^s \circ F) + I_{s,\beta} + II_{s,\beta} + III_{s,\beta} + IV_{s,\beta}.
$$
By Proposition \ref{deriapp}, the first term satisfies
\begin{align*}
\left| I_{s,\beta} \right|_{\scriptscriptstyle \mathcal{B}^{s - \beta + \gamma}_{1,1}} 
&= \left| z_P O_{P_\star} |P_\star|^{s+\gamma} D^\beta  1_{P_\star}^{-1}  \right|_{\scriptscriptstyle \mathcal{B}^{s - \beta + \gamma}_{1,1}} \nonumber \\
&\leq |z_P| O_{P_\star} (\ln |P_\star|)^{(s+\gamma)_0}  |P_\star|^{(s+\gamma)_-}    \leq C |\phi|_{\scriptscriptstyle \mathcal{B}^s_{1,1}}.
\end{align*}
For the second term, \change{$s-\beta+\gamma> 0$, $ s+\gamma > 0$} 
\begin{align}
\left| II_{s,\beta} \right|_{\scriptscriptstyle \mathcal{B}^{s - \beta + \gamma}_{1,1}} 
&= \left| z_P O_{P_\star} 1_{P_\star}^{s - \beta + \gamma - 1} \right|_{\scriptscriptstyle \mathcal{B}^{s - \beta + \gamma}_{1,1}} \nonumber \\ 
&\leq |z_P| O_{P_\star} (\ln |P_\star|)^{(s - \beta + \gamma )_0}  |P_\star|^{(s - \beta + \gamma )_-} \leq C |\phi|_{\scriptscriptstyle \mathcal{B}^s_{1,1}}.\nonumber
\end{align}
For the third term,
\begin{align}
III_{s,\beta}
&= z_P O_{P_\star} \phi_{P_\star}^{s - \beta + \gamma}
+ z_P O_{P_\star} \phi_{P_\star}^{s - \beta} c_{J_+} \phi_{J_+}^{\gamma + 1} \notag \\
&= z_P O_{P_\star} \phi_{P_\star}^{s - \beta + \gamma}
+ z_P O_{P_\star} |P_\star|^{(s - \beta)_+ +  \hat{\epsilon}} \left( \frac{|J_+|}{|P_\star|} \right)^{-(s - \beta)_- + \hat{\epsilon}  + 1} \phi_{J_+}^{\gamma + (s - \beta)_- -  \hat{\epsilon}}. \nonumber
\end{align}
We estimate the two components:
\begin{align*}
\left| z_P O_{P_\star} \phi_{P_\star}^{s - \beta + \gamma} \right|_{\scriptscriptstyle \mathcal{B}^{\gamma + s - \beta}_{1,1}} 
\leq |z_P| O_{P_\star} \leq C |\phi|_{\scriptscriptstyle \mathcal{B}^s_{1,1}},
\end{align*}
and
\begin{align*}
&\left| z_P O_{P_\star} |P_\star|^{(s - \beta)_+ + \hat{\epsilon}} \left( \frac{|J_+|}{|P_\star|} \right)^{-(s - \beta)_- + \hat{\epsilon} + 1} \phi_{J_+}^{\gamma + (s - \beta)_- - \hat{\epsilon}} \right|_{\scriptscriptstyle \mathcal{B}^{\gamma + (s - \beta)_- - \hat{\epsilon}}_{1,1}} \\
&\leq |z_P| O_{P_\star} |P_\star|^{(s - \beta)_+ + \hat{\epsilon}} \leq C |\phi|_{\scriptscriptstyle \mathcal{B}^s_{1,1}}.
\end{align*}
Finally, for the fourth term,
\begin{align*}
IV_{s,\beta}
&= z_P O_{P_\star} |P_\star|^{s - \beta + \gamma - 1} c_{J_+} \phi_{J_+}^{\gamma + 1} \notag \\
&= z_P O_{P_\star} |P_\star|^{(s - \beta)_+ + \hat{\epsilon} + \gamma} \left( \frac{|J_+|}{|P_\star|} \right)^{-(s - \beta)_- + 1 + \hat{\epsilon}} \phi_{J_+}^{\gamma + (s - \beta)_- - \hat{\epsilon}}.
\end{align*}
We conclude that
$$
\left| IV_{s,\beta} \right|_{\scriptscriptstyle \mathcal{B}^{\gamma + (s - \beta)_- - \hat{\epsilon}}_{1,1}} \leq C |\phi|_{\scriptscriptstyle \mathcal{B}^s_{1,1}}.
$$
\end{proof}

\begin{proof}[Proof of Leibniz Rule I (Theorem \ref{leibnitz})]   Let $\psi = z_0 1_I + z_P \phi_P^s$. 
Note that 
\begin{equation} \label{zeroterm} D^\beta (z_0 1_I \cdot h)- D^\beta (z_0 1_I )\cdot h= z_0 D^\beta h \in   \mathcal{B}^{\gamma-\beta}_{\infty,\infty}\subset \mathcal{B}^{\gamma-\beta -\epsilon}_{1,1}\end{equation} 
for every $\epsilon > 0$.

By Theorem \ref{fbeta}, we have for every block $\omega_P$
$$
\omega_P 1_P h = \omega_P M_P 1_P + \omega_P c_P \phi_P^{\gamma+1} + \omega_P c_{J_+} \phi_{J_+}^{\gamma+1},
$$
where $M_P = \frac{1}{|P|} \int_P h \, dm$ and $|c_J| \leq |h|_{\scriptscriptstyle \mathcal{B}^\gamma_{\infty,\infty}}$.
Then
\begin{align}\label{eqws} 
&z_P D^\beta(\phi_P^s) \cdot h
 = z_P M_P \phi_P^{s-\beta}
+ z_P c_P |P|^{s - \beta + \gamma + 1} 1_{P_{\mathrm{ch}}}^{-2}
+ z_P \phi_P^{s-\beta} c_{J_+} \phi_{J_+}^{\gamma+1}   \\
& = z_P M_P \phi_P^{s-\beta}
+ \underbrace{1_P^{s-\beta+\gamma-1} O_P z_{\mathop{\mathrm{par}}(P)} |h|_{\scriptscriptstyle \mathcal{B}^\gamma_{\infty,\infty}}}_{\scriptscriptstyle I_{s,\beta}}
+ \underbrace{z_P O_P |P|^{s-\beta-1} O_{J_+} |h|_{\scriptscriptstyle \mathcal{B}^\gamma_{\infty,\infty}} \phi_{J_+}^{\gamma+1}}_{\scriptscriptstyle II_{s,\beta}}.  \nonumber 
\end{align}
We estimate each term: \change{$s - \beta + \gamma> 0$}
\begin{align*}
\left| z_P M_P \phi_P^{s-\beta} \right|_{\scriptscriptstyle \mathcal{B}^{s-\beta}_{1,1}}
&= \left| z_P \phi_P^{s-\beta} O_P |h|_{\scriptscriptstyle \mathcal{B}^\gamma_{\infty,\infty}} \right|_{\scriptscriptstyle \mathcal{B}^{s-\beta}_{1,1}} 
\leq C |\phi|_{\scriptscriptstyle \mathcal{B}^s_{1,1}} |h|_{\scriptscriptstyle \mathcal{B}^\gamma_{\infty,\infty}}, \\
\left| 1_P^{s - \beta + \gamma - 1} O_P z_{\mathop{\mathrm{par}}(P)} \right|_{\scriptscriptstyle \mathcal{B}^{s - \beta + \gamma}_{1,1}} 
&\leq C |\phi|_{\scriptscriptstyle \mathcal{B}^s_{1,1}}, \\
\left| z_P O_P |P|^{s - \beta - 1} O_{J_+} \phi_{J_+}^{\gamma + 1} \right|_{\scriptscriptstyle \mathcal{B}^{\gamma + (s - \beta)_- - \hat{\epsilon}}_{1,1}} 
&\leq |z_P| O_P |P|^{(s - \beta)_+ + \hat{\epsilon}} \leq C |\phi|_{\scriptscriptstyle \mathcal{B}^s_{1,1}}.
\end{align*}
On the other hand, we also compute
\begin{align*}
z_P \phi_P^s \cdot h 
&= z_P M_P \phi_P^s 
+ z_P c_P |P|^{s + \gamma + 1} 1_{P_{\mathrm{ch}}}^{-2}
+ z_P |P|^{s - 1} c_{J_+} O_{J_+} \phi_{J_+}^{1 + \gamma},
\end{align*}
and therefore,
\begin{align}\label{eqws2}
&z_P D^\beta(|P|^s \phi_P \cdot h) \notag\\
&= z_P M_P \phi_P^{s - \beta} 
+ \underbrace{z_P c_P |P|^{s + \gamma + 1} D^\beta 1_{P_{\mathrm{ch}}}^{-2}}_{\scriptscriptstyle III_{s,\beta}}
+ \underbrace{z_P |P|^{s - 1} c_{J_+} O_{J_+} \phi_{J_+}^{1 + \gamma - \beta}}_{\scriptscriptstyle IV_{s,\beta}}.
\end{align}
By Proposition \ref{deriapp}, since $\gamma + s > 0$ we have \change{$\gamma +s > 0$} 
\begin{align}
\left| III_{s,\beta} \right|_{\scriptscriptstyle \mathcal{B}^{\gamma + s - \beta}_{1,1}} 
&\leq |z_P| |P|^{s + \gamma + 1} \left| D^\beta 1_{P_{\mathrm{ch}}}^{-2} \right|_{\scriptscriptstyle \mathcal{B}^{\gamma + s - \beta}_{1,1}} |h|_{\scriptscriptstyle \mathcal{B}^\gamma_{\infty,\infty}} \notag \\
&\leq |z_P| |P|^{s + \gamma + 1} |P_{\mathrm{ch}}|^{-\gamma - s - 1} |h|_{\scriptscriptstyle \mathcal{B}^\gamma_{\infty,\infty}} 
\leq C |\phi|_{\scriptscriptstyle \mathcal{B}^s_{1,1}} |h|_{\scriptscriptstyle \mathcal{B}^\gamma_{\infty,\infty}}. \notag
\end{align}
Finally,
\begin{align*}
\left| IV_{s,\beta} \right|_{\scriptscriptstyle \mathcal{B}^{\gamma  + s_-  - \beta - \tilde{\epsilon}}_{1,1}} &\leq \left| z_P |P|^{s - 1} O_{J_+} \phi_{J_+}^{1 + \gamma - \beta} \right|_{\scriptscriptstyle \mathcal{B}^{\gamma +s_- -\beta - \tilde{\epsilon}}_{1,1}} \\
&= \left| z_P |P|^{s - 1} O_{J_+} |J_+|^{1 -s_- +  \tilde{\epsilon}} \phi_{J_+}^{\gamma + s_- - \beta -  \tilde{\epsilon}} \right|_{\scriptscriptstyle \mathcal{B}^{\gamma  +s_- - \beta -  \tilde{\epsilon}}_{1,1}} \notag \\
&\leq |z_P| |P|^{s_+ +  \tilde{\epsilon}} \leq C |\phi|_{\scriptscriptstyle \mathcal{B}^s_{1,1}}.
\end{align*}
We thus define the remainder as
$$
R_L(\psi) = z_0 D^\beta h -I_{s,\beta} - II_{s,\beta} + III_{s,\beta} + IV_{s,\beta}.
$$
\end{proof}

\begin{proof}[Proof of Leibniz Rule II (Theorem \ref{leibnitz1})]
Every function $\psi \in \mathcal{B}_{\infty,\infty}^s$ can be written as
$$
\psi = z_0 1_I + z_P \phi_P^{s+1},
$$
where
$$
|\psi|_{\scriptscriptstyle \mathcal{B}^s_{\infty,\infty}} = |z_0| + \sup_k \sup_{P \in \mathcal{P}^k} |z_P|.
$$
we deal with with the term $D^\beta (z_01_I\cdot h)$ as in (\ref{zeroterm}). To the remainder terms, apply the same decomposition as in equation (\ref{eqws}), replacing $s$ by $s+1$ in the expression for $(D^\beta |P|^{s+1} \phi_P) \cdot h$. Then, the first term satisfies
\begin{align*}
\left| I_{s+1,\beta} \right|_{\scriptscriptstyle \mathcal{B}^{s - \beta + \gamma}_{\infty,\infty}} 
= \left| 1_P^{s - \beta + \gamma} O_P |\phi|_{\scriptscriptstyle \mathcal{B}^s_{\infty,\infty}} |h|_{\scriptscriptstyle \mathcal{B}^\gamma_{\infty,\infty}} \right|_{\scriptscriptstyle \mathcal{B}^{s - \beta + \gamma}_{\infty,\infty}} 
\leq C |\phi|_{\scriptscriptstyle \mathcal{B}^s_{\infty,\infty}} |h|_{\scriptscriptstyle \mathcal{B}^\gamma_{\infty,\infty}}.
\end{align*}

For the second term,
\begin{align*}
\left| II_{s+1,\beta} \right|_{\scriptscriptstyle \mathcal{B}^{(s - \beta)_- - \epsilon + \gamma}_{\infty,\infty}} 
&= \left| \phi_J^{\gamma + 1} O_{P_-} |P_-|^{s - \beta} |\phi|_{\scriptscriptstyle \mathcal{B}^s_{\infty,\infty}} |h|_{\scriptscriptstyle \mathcal{B}^\gamma_{\infty,\infty}} \right|_{\scriptscriptstyle \mathcal{B}^{(s - \beta)_- + \gamma - \epsilon}_{\infty,\infty}} \notag \\
&= \left| \phi_J^{\gamma + (s - \beta)_- - \epsilon + 1} O_J |\phi|_{\scriptscriptstyle \mathcal{B}^s_{\infty,\infty}} |h|_{\scriptscriptstyle \mathcal{B}^\gamma_{\infty,\infty}} \right|_{\scriptscriptstyle \mathcal{B}^{(s - \beta)_- + \gamma - \epsilon}_{\infty,\infty}} \notag \\
&\leq C |\phi|_{\scriptscriptstyle \mathcal{B}^s_{\infty,\infty}} |h|_{\scriptscriptstyle \mathcal{B}^\gamma_{\infty,\infty}}.
\end{align*}
We now consider the   same decomposition as  in equation (\ref{eqws2}), replacing $s$ by $s+1$:
\begin{align*}
III_{s+1,\beta}
&= O_P |P|^{s + \gamma + 2} |1_{P_{\mathrm{ch}}}|^{-1} O_{W_-} \phi_{W_-}^{-\beta} |\phi|_{\scriptscriptstyle \mathcal{B}^s_{\infty,\infty}} |h|_{\scriptscriptstyle \mathcal{B}^\gamma_{\infty,\infty}} \notag \\
&= O_P |P|^{s + \gamma + 1} O_{W_-} \phi_{W_-}^{-\beta} |\phi|_{\scriptscriptstyle \mathcal{B}^s_{\infty,\infty}} |h|_{\scriptscriptstyle \mathcal{B}^\gamma_{\infty,\infty}} \notag \\
&= O_W \phi_W^{-\beta} O_{P_+} |P_+|^{s + \gamma + 1} |\phi|_{\scriptscriptstyle \mathcal{B}^s_{\infty,\infty}} |h|_{\scriptscriptstyle \mathcal{B}^\gamma_{\infty,\infty}} \notag \\
&= O_W \phi_W^{s - \beta + \gamma + 1} |\phi|_{\scriptscriptstyle \mathcal{B}^s_{\infty,\infty}} |h|_{\scriptscriptstyle \mathcal{B}^\gamma_{\infty,\infty}}.
\end{align*}
Hence,
\begin{align*}
\left| III_{s+1,\beta} \right|_{\scriptscriptstyle \mathcal{B}^{s - \beta + \gamma}_{\infty,\infty}} 
\leq C |\phi|_{\scriptscriptstyle \mathcal{B}^s_{\infty,\infty}} |h|_{\scriptscriptstyle \mathcal{B}^\gamma_{\infty,\infty}}.
\end{align*}
Finally, we estimate the last term: 
\begin{align*}
\left| IV_{s+1,\beta} \right|_{\scriptscriptstyle \mathcal{B}^{s_-+ \gamma - \beta -\overline{\epsilon} }_{\infty,\infty}} 
&= \left| z_P |P|^s c_{J_+}  O_{J_+} \phi_{J_+}^{1 + \gamma - \beta} \right|_{\scriptscriptstyle \mathcal{B}^{s_-+ \gamma - \beta -\overline{\epsilon} }_{\infty,\infty}}   \\
&= \left| c_J O_J \phi_J^{1 + \gamma - \beta} z_{P_-} |P_-|^s \right|_{\scriptscriptstyle \mathcal{B}^{s_-+ \gamma - \beta -\overline{\epsilon} }_{\infty,\infty}} \\
&= \left|   O_J \phi_J^{1 + \gamma - \beta} O_{P_-}  |P_-|^s |h|_{\scriptscriptstyle \mathcal{B}^\gamma_{\infty,\infty}} |\phi|_{\scriptscriptstyle \mathcal{B}^s_{\infty,\infty}} \right|_{\scriptscriptstyle \mathcal{B}^{s_-+ \gamma - \beta -\overline{\epsilon} }_{\infty,\infty}} \\
&= \left|   O_J \phi_J^{1 + s_-+ \gamma - \beta -\overline{\epsilon}  } O_{P_-} |\log  |J||^{s_0} |J|^{\overline{\epsilon}  } |h|_{\scriptscriptstyle \mathcal{B}^\gamma_{\infty,\infty}} |\phi|_{\scriptscriptstyle \mathcal{B}^s_{\infty,\infty}} \right|_{\scriptscriptstyle \mathcal{B}^{s_-+ \gamma - \beta -\overline{\epsilon} }_{\infty,\infty}} \\
&\leq C |h|_{\scriptscriptstyle \mathcal{B}^\gamma_{\infty,\infty}} |\phi|_{\scriptscriptstyle \mathcal{B}^s_{\infty,\infty}}.
\end{align*}

\end{proof}

\subsection{Connection with the Fractional Laplacian}

The fractional Laplacian (see Di Nezza, Palatucci, and Valdinoci \cite{nezza}) is a fundamental example of a non-local pseudo-differential operator in classical analysis. It arises in various contexts, including probability theory, partial differential equations, and harmonic analysis. There are several equivalent definitions of the fractional Laplacian, depending on the functional setting and the level of regularity of the function.

For rapidly decaying smooth functions $u \in C^\infty(\mathbb{R}^n)$ and for any $s > 0$, the fractional Laplacian of order $s$ can be defined via the singular integral representation:
$$
(-\Delta)^s u(x) = C_{n,s} \, \mathrm{P.V.} \int_{\mathbb{R}^n} \frac{u(x) - u(y)}{|x - y|^{n + 2s}} \, dy,
$$
where $C_{n,s}$ is an explicit normalization constant depending on the dimension $n$ and the order $s$, and $\mathrm{P.V.}$ denotes the Cauchy principal value.

The formula above highlights the non-local nature of $(-\Delta)^s$, as the value of the operator at $x$ depends on the values of $u$ at all other points $y \in \mathbb{R}^n$.

In our setting, the operator $D^\beta$ introduced earlier may be viewed as a generalization of the fractional Laplacian, adapted to the particular geometric or functional framework under consideration.

\begin{proposition}  Let $\beta > 0$. There is $\Cll{rt}>1$ such that for every $
Q \in \mathcal{P}$ there is  $\Cll{yy}(Q)$ such that $  \Crr{yy}(Q)\in (1,\Crr{rt})$ and 
\begin{equation} \Crr{yy}(Q) (D^\beta \phi_Q)(x)= \int \frac{\phi_Q(x)-\phi_Q(y)}{d(x,y)^{1+\beta}} \ dm(y).\end{equation}
\end{proposition} 
\begin{proof} Let $P^n(x)$ such that $x\in P^n(x)\in \mathcal{P}^n$. Suppose that $Q\in \mathcal{P}^{n_0}$.  
Then 
\begin{align*}  \int \frac{\phi_Q(x)-\phi_Q(y)}{d(x,y)^{1+\beta}} \ dm(y)  &= \sum_{n\leq n_0} \frac{1}{|P^n(x)|^{1+\beta}} \int_{P^n(x)\setminus P^{n+1}(x)} \phi_Q(x)-\phi_Q(y)  \ dm(y)\\
&=  \frac{\phi_Q(x)}{|Q|^\beta}+  \sum_{\substack{n<  n_0\\ Q\subset P^{n+1}(x) }} \phi_Q(x) \frac{|P^n(x)| -|P^{n+1}(x)|}{|P^n(x)|^{1+\beta}}\\
&=  \frac{\phi_Q(x)}{|Q|^\beta}\Big( 1 +   \sum_{ n<  n_0 } \frac{|Q|^\beta}{|P^n(Q)|^\beta} \big( 1- \frac{|P^{n+1}(Q)|}{|P^n(Q)|} \big) \Big)\\
&= \frac{\phi_Q(x)}{|Q|^\beta} \Crr{yy}(Q).
\end{align*}
\end{proof}

\subsection{Holomorphy and Schauder bases}  
The following is a simple and useful criterion for holomorphy in Banach spaces with a Schauder basis.

\begin{proposition}
Let $X$ be a Banach space with a Schauder basis $(e_n)_{n\ge 1}$, and let
$f_n:D\to\mathbb{C}$ be holomorphic functions on a disc $D\subset\mathbb{C}$.
Assume that, for every $\lambda\in D$, the series
\[
F(\lambda)\;=\;\sum_{n=1}^{\infty} f_n(\lambda)\,e_n
\]
converges in $X$ and that $F$ is locally bounded on $D$ (i.e. bounded on every compact subset of $D$).
Then $F:D\to X$ is holomorphic (in the Banach-space sense). 
 
\end{proposition}

\begin{proof}
Let $P_N:X\to X$ denote the $N$-th partial-sum projection associated with the basis,
\[
P_N\Big(\sum_{n=1}^{\infty} a_n e_n\Big)\;=\;\sum_{n=1}^{N} a_n e_n .
\]
For a Schauder basis, these projections are uniformly bounded:
\[
K\;:=\;\sup_{N\ge 1}\,\|P_N\|\;<\;\infty .
\]
(See \cite[\S1.3]{AlbiacKalton2016} )

Define $S_N(\lambda):=P_N(F(\lambda))=\sum_{n=1}^{N} f_n(\lambda)e_n$.
Each $S_N$ is holomorphic on $D$ (finite sum of holomorphic functions), and for every compact $K\Subset D$ we have
\[
\sup_{\lambda\in K}\|S_N(\lambda)\|
\;\le\; \|P_N\|\ \sup_{\lambda\in K}\|F(\lambda)\|
\;\le\; K\,\sup_{\lambda\in K}\|F(\lambda)\|
\qquad \text{(uniformly in $N$)} .
\]
Fix $\varphi\in X^*$. Then $\varphi\circ S_N$ are holomorphic scalar functions, uniformly bounded on every compact $K\Subset D$, and
$\varphi\circ S_N(\lambda)\to \varphi\circ F(\lambda)$ pointwise (because $S_N(\lambda)\to F(\lambda)$ in $X$).
By Vitali–Montel, $\varphi\circ S_N \to \varphi\circ F$ \emph{uniformly on compact subsets} of $D$, hence $\varphi\circ F$ is holomorphic on $D$.

Since $\varphi\circ F$ is holomorphic for every $\varphi\in X^*$ and $F$ is locally bounded, the
\emph{weak–strong holomorphy criterion} in Banach spaces implies that $F$ is holomorphic
(see \cite[Thm.\ 2.15]{Mujica1986} )

\end{proof}

\bibliographystyle{abbrvurl}
\bibliography{bibliografiab}

\end{document}